\documentclass[11pt]{amsart} 

\usepackage{amsmath,amssymb,amsthm,amsfonts,verbatim}
\usepackage{enumerate}
\usepackage{color}
\usepackage{graphicx}

\usepackage{tikz-cd}

\usepackage[utf8]{inputenc}

\theoremstyle{plain}
\newtheorem*{cor}{Corollary}
\newtheorem{theorem}{Theorem}[section]

\newtheorem{proposition}[theorem]{Proposition}

\newtheorem{lemma}[theorem]{Lemma}

\newtheorem{remark}[theorem]{Remark}

\newtheorem{corollary}[theorem]{Corollary}
\theoremstyle{definition}
\newtheorem{definition}[theorem]{Definition}

\newcommand{\nc}{\newcommand}
\nc{\dmo}{\DeclareMathOperator}

\nc{\QQ}{\mathbb{Q}}
\nc{\RR}{\mathbb{R}}
\nc{\NN}{\mathbb{N}}
\nc{\RP}{\mathbb{RP}^1}
\nc{\ZZ}{\mathbb{Z}}
\nc{\CC}{\mathbb{C}}
\nc{\cS}{\mathcal{S}}
\nc{\iso}{\cong}
\dmo{\Mod}{Mod}
\dmo{\Ig}{\mathcal{I}_g}
\dmo{\Span}{span}
\dmo{\Diff}{Diff}
\dmo{\Homeo}{Homeo}
\dmo{\dist}{dist}
\dmo\BDiff{BDiff}
\dmo\SO{SO}
\dmo\slide{sl}
\dmo\im{im}
\dmo\id{id}
\dmo\Fix{Fix}
\dmo\Stab{Stab}
\dmo\Mcg{Mcg}
\dmo\Out{Out}
\dmo\Aut{Aut}
\dmo{\Hg}{\mathcal{H}_g}
\dmo{\Han}{\mathcal{H}}
\dmo{\Tg}{\mathcal{T}_g}

\newcommand\CP{\mathcal{P}^{\mathrm{nm}}}
\newcommand\CPT{\mathcal{M}^{\mathrm{nm}}}
\dmo\W{W}
\dmo\A{A}

\renewcommand{\epsilon}{\varepsilon}
\nc{\coloneq}{\mathrel{\mathop:}\mkern-1.2mu=}
\nc{\margin}[1]{\marginpar{\scriptsize #1}}
\nc{\para}[1]{\bigskip\noindent\textbf{#1}}

\begin{document}
\title{The Geometry of the Handlebody Groups II: Dehn functions}
\author{Ursula Hamenst\"adt and Sebastian Hensel}
\date{27.4.2018\\Ursula Hamenstädt was supported by the ERC grant ``Moduli''.}
\begin{abstract}
  We show that the Dehn function of the handlebody group is
  exponential in any genus $g\geq 3$. On the other hand, we show that
  the handlebody group of genus $2$ is cubical, biautomatic, and therefore has
  a quadratic Dehn function.
\end{abstract}
\maketitle

\section{Introduction}
\label{sec:intro}

This article is concerned with the word geometry of the
\emph{handlebody group} $\Han_g$, i.e. the mapping class group of a
handlebody of genus $g$. The core motivation to study this group is
twofold.  On the one hand, handlebodies are basic building blocks for
three-manifolds -- namely, for any closed \mbox{$3$--manifold} $M$ there is a
$g$ so that $M$ can be obtained by gluing two genus $g$ handlebodies
$V, V'$ along their boundaries with a homeomorphism $\varphi$.  Any
topological property of $M$ is then determined by the gluing map
$\varphi$.  One of the difficulties in extracting this information is
that $\varphi$ is by no means unique. In fact, modifying it on either
side by a homeomorphism which extends to $V$ or $V'$ does not change
$M$. In this sense, the handlebody group encodes part of the
non-uniqueness of the description of a $3$--manifold via a Heegaard
splitting.

\smallskip The other motivation stems from geometric group theory, and
it is the more pertinent for the current work. Identify
the boundary surface of $V_g$ with a surface $\Sigma_g$ of genus $g$.
Then there is a restriction homomorphism of $\Han_g$ into the surface
mapping class group $\mathrm{Mcg}(\Sigma_g)$, and it is not hard
to see that it is injective. On the other hand,
considering the action of homeomorphisms of $V_g$ on the fundamental
group $\pi_1(V_g) = F_g$ gives rise to a surjection of $\Han_g$ onto
$\Out(F_g)$. The handlebody group is thus immediately related to two
of the most studied groups in geometric group theory.

But from a geometric perspective, neither of these relations is
simple: in previous work \cite{HH1} we showed that the inclusion of
$\Han_g$ into $\mathrm{Mcg}(\Sigma_g)$ is exponentially distorted for
any genus $g\geq 2$. Furthermore, a result by McCullough shows that
the kernel of the surjection $\Han_g \to \Out(F_g)$ is infinitely generated.

In particular, there is no a priori reason to expect that $\Han_g$ shares
geometric properties with either surface mapping class groups or outer
automorphism groups of free groups.

\smallskip
The first main result of this paper shows that the geometry of $\Han_g$ for
$g\geq 3$ seems to share geometric features with the (outer) automorphism group of
a free group. In this result, we slightly extend our perspective and also consider
handlebodies $V_{g,1}$ of genus $g$ with a marked point and their handlebody group $\Han_{g,1}$.
We show.
\begin{theorem}\label{thm:intro1}
  The Dehn function of $\Han_g$ and $\Han_{g,1}$ is exponential for
  any $g \geq 3$.
\end{theorem}
The Dehn function of a group is a combinatorial isoperimetric function,
and it is a geometric measure for the difficulty of the word problem
(see Section~\ref{sec:prelim-dehn} for details and a formal
definition). Theorem~\ref{thm:intro1} should be contrasted with the
situation in the surface mapping class group -- by a theorem of Mosher
\cite{Mosher-automatic}, these groups are automatic and therefore have
quadratic Dehn functions. On the other hand, Bridson and Vogtmann
showed that $\Out(F_g)$ has exponential Dehn function for $g\geq 3$
\cite{BV-Dehn, BV-Dehn2, HV-Dehn}.

Mapping class groups of small complexity are known to have
properties not shared with properties of mapping class groups
of higher complexity. For example, 
the mapping class group $\mathrm{Mcg}(\Sigma_2)$
of a surface of genus is a $\mathbb{Z}/2\mathbb{Z}$-extension of
the mapping class group of a sphere with 6 punctures.
This implies among others that the group 
virtually
surjects onto $\mathbb{Z}$, a property which is
not known for higher genus. 
On the other hand, the group ${\rm Out}(F_2)$ is just the
full linear group $GL(2,\mathbb{Z})$. Similarly, it is known that the
genus $2$ handlebody group surjects onto $\mathbb{Z}$ as well \cite{IS}.

Our second goal is to add to these results by showing that the
handlebody group $\Han_2$ has properties not shared by or unknown
for handlebody groups of higher genus. 

\begin{theorem}\label{thm:intro2}
  The group $\Han_2$ admits a proper cocompact action on a
  ${\rm CAT}(0)$ cube complex.
\end{theorem}

As an immediate corollary, using Corollary 8.1 of \cite{Swiatkowski}
and Proposition 1 of \cite{CMV}), we obtain among others
that the genus bound in Theorem \ref{thm:intro1} is optimal. 

\begin{cor}\label{cor}
  The group $\Han_2$ is biautomatic, in particular it has quadratic Dehn
  function, and it has the Haagerup property.
\end{cor}

In order to prove Theorem~\ref{thm:intro1}, we recall from
previous work \cite{Transactions} that the Dehn
function of handlebody groups is at most exponential, and it therefore
suffices to exhibit a family of cycles which requires exponential area
to fill (compared to their lengths). These cycles will be lifted from
cycles in automorphism groups of free groups used by 
Bridson and Vogtmann \cite{BV-Dehn}. This
construction occupies Section~\ref{sec:exponential}.

The proof of Theorem~\ref{thm:intro2} is more involved, and relies on
constructing and studying a suitable geometric model for the genus $2$
handlebody group. In Section~\ref{sec:waves} we describe in detail the
intersection pattern of disk-bounding curves in a genus $2$ handlebody.
The model of $\Han_2$ is built in two steps: in
Section~\ref{sec:complexes} we construct a tree on which $\Han_2$ acts
and in Section~\ref{sec:consequences-2} we then use this tree to build
our cubical model for $\Han_2$ and prove Theorem~\ref{thm:intro2}.
We also discuss some additional geometric consequences.

\section{Preliminaries}
\label{sec:prelims}

\subsection{Handlebody Groups}
Let $V$ be a handlebody of genus $g\geq 2$. We identify the boundary of
$V$ once and for all with a surface $\Sigma$ of genus
$g$. Restrictions of homeomorphisms of $V$ to the boundary then
induce a restriction map
\[ r:\mathrm{Mcg}(V) \to \mathrm{Mcg}(\Sigma). \]
It is well-known that this map is injective, and we call its image the
\emph{handlebody group} $\Han_g$.

When we consider a handlebody with a marked point $p$, we will always assume
that the marked point is contained in the boundary. We then get a map
\[ r:\mathrm{Mcg}(V,p) \to \mathrm{Mcg}(\Sigma,p) \]
whose image is the handlebody group $\Han_{g,1}$.

\smallskip We also need some special elements of the handlebody
group. We denote by $T_\alpha \in \Mcg(\Sigma)$ the positive (or left) \emph{Dehn twist
  about $\alpha$} (compare \cite[Chapter~3]{Primer}). Dehn twists $T_\alpha$ are
elements of the handlebody group exactly if the curve $\alpha$ is a
\emph{meridian}, i.e. a curve which is the boundary of an embedded
disk $D \subset V$.

We have the following standard lemma which gives rise to another
important class of elements. 
\begin{lemma}\label{lem:annulus-twist}
  Suppose that $\alpha, \beta, \delta$ are three disjoint simple closed curves
  on $\Sigma$ which bound a pair of pants on $\Sigma$. Suppose that $\delta$
  is a meridian. Then the product
  \[ T_\alpha T_\beta^{-1} \]
  is an element of the handlebody group.
\end{lemma}
\begin{proof}
  Let $P$ be the pair of pants with $\partial P = \alpha \cup \beta
  \cup \delta$, and let $D$ be the disk bounded by $\delta$. Then $P
  \cup D$ is an embedded annulus in $V$ whose two boundary curves are
  $\alpha$ and $\beta$. By applying a small isotopy we may then assume
  that there is a properly embedded annulus $A$ whose boundary curves
  are $\alpha, \beta$.

  Consider the homeomorphism $F$ of $V$
  which is a twist about $A$. To be more precise, consider a regular
  neighborhood $U$ of $A$ of the form
  \[ U = [0, 1] \times A = [0,1]\times S^1 \times[0,1].\] The
  homeomorphism $F$ is defined to be the standard Dehn twist on each
  annular slice $[0,1]\times S^1 \times \{t\} \subset U$. This map
  restricts to the identity on $\{0, 1\}\times A$, and thus extends to
  a homeomorphism of $V$. It restricts on the boundary of $V$ to the
  desired element, finishing the proof.
\end{proof}

The final type of elements we need are \emph{point-push maps}. Recall (e.g.
from \cite{Primer}) the \emph{Birman exact sequence}
\[ 1 \to \pi_1(\Sigma, p) \to \mathrm{Mcg}(\Sigma, p) \to
  \mathrm{Mcg}(\Sigma) \to 1.\]
The image of $\pi_1(\Sigma, p) \to \mathrm{Mcg}(\Sigma, p)$ is the point
pushing subgroup. We need three facts about these mapping classes, all of
which are well-known, and are fairly immediate from the definition.
\begin{lemma}\label{lem:point-push-facts}
\begin{enumerate}[i)]
\item The point-pushing subgroup is contained in $\Han_{g,1}$ 
  (compare \cite[Section~3]{HH1}).
\item If $\gamma\in\pi_1(\Sigma, p)$ is simple, then the point push
  about $\gamma$ is a product $T_\alpha T_\beta^{-1}$ of two Dehn
  twists, where $\alpha, \beta$ are the two simple closed curves
  obtained by pushing $\gamma$ off itself to the left and right,
  respectively (compare \cite[Fact~4.7]{Primer}).
\item The point push about $\gamma$ acts on $\pi_1(\Sigma,p)$ as conjugation
  by $\gamma$. Similarly, it acts on $\pi_1(V, p)$ as conjugation by
  the image of $\gamma$ in $\pi_1(V,p)$ (compare \cite[Discussion in Section~4.2.1]{Primer}).
\end{enumerate}
\end{lemma}

\subsection{Dehn functions}
\label{sec:prelim-dehn}
Consider a finitely presented group $G$ with a fixed finite presentation
$\langle S | R \rangle$. A word $w$ in $S$ (or, alternatively, an element of
the free group $F(S)$ on the set $S$) is trivial in $G$ exactly if
$w$ can be written (in $F(S)$) as a product 
\[ w = \prod_{i=1}^n x_i r_i x^{-1}_i \]
for elements $r_i \in R$ and $x_i \in F(S)$. We define the \emph{area of $w$}
as the minimal $n$ for which such a description is possible. The
\emph{Dehn function} is the function
\[ D(n) = \sup\{ \mathrm{area}(w) \;|\; l(w) = n \} \]
where $l(w)$ denotes the length of the word $w$ (alternatively, the word norm
in $F(S)$).

The Dehn function depends on the choice of the presentation, but its
growth type does not (see e.g. \cite{Alonso}). We employ the convention that
products in mapping class groups are compositions (i.e. the rightmost
mapping classes are applied first).

\subsection{Annular subsurface projections}
\label{sec:annular-projections}
In this subsection we briefly recall subsurface projections into annular regions,
as defined in \cite{MM2}, Section~2.4.

Let $A = S^1 \times [0,1]$ be a closed annulus. Recall that the
\emph{arc graph} $\mathcal{A}(A)$ of the annulus $A$ is the graph
whose vertices correspond to embedded arcs which connect the two boundary
circles $S^1\times\{0\}, S^1\times\{1\}$, up to homotopy fixing the
endpoints. Two such vertices are joined by an edge if the
corresponding arcs are disjoint except possibly at the endpoints (up
to homotopy fixing the endpoints). It is shown in Section~2.4 of
\cite{MM2} that the resulting graph is quasi-isometric to the integers.

Now consider a surface $\Sigma$ of genus at least two. Fix once and for all a hyperbolic
metric on $\Sigma$. If $\alpha$ is any simple closed curve on
$\Sigma$, let $\Sigma_\alpha \to \Sigma$ be the annular cover
corresponding to $\alpha$, i.e. the cover homeomorphic to an (open)
annulus to which $\alpha$ lifts with degree $1$. By pulling back the
hyperbolic metric from $\Sigma$ to $\Sigma_\alpha$, we obtain a
hyperbolic metric on $\Sigma_\alpha$. This allows us to add two
boundary circles at infinity which compactify $\Sigma_\alpha$ to a
closed annulus $\widehat{\Sigma_\alpha}$.

If $\beta$ is a simple closed curve on $\Sigma$, then any lift
$\hat{\beta}$ of $\beta$ to $\Sigma_\alpha$ has well-defined endpoints
at infinity (for example, since this is true for lifts to the
universal cover). In addition, if $\beta$ has an essential intersection with $\alpha$, there is at least one lift
$\hat{\beta}$ of $\beta$ to $\Sigma_\alpha$ which connects the two
boundary circles of $\Sigma_\alpha$.  Such a lift $\hat{\beta}$ has
well-defined endpoints at infinity in $\widehat{\Sigma_\alpha}$, and
so it defines a vertex in $\mathcal{A}(\widehat{\Sigma_\alpha})$.  We
define the projection $\pi_\alpha(\beta) \subset
\mathcal{A}(\widehat{\Sigma_\alpha})$ to be the set of all such
lifts. Since $\beta$ is simple, this is a (finite) subset of diameter
one.

Observe that if $\beta'$ is freely
homotopic to $\beta$, then any lift of $\beta'$ is homotopic to a lift
of $\beta$ with the same endpoints at infinity. Hence, the projection
$\pi_\alpha(\beta)$ depends only on the free homotopy class of $\beta$.

If $\beta$ is disjoint from $\alpha$, the projection $\pi_\alpha(\beta)$ is undefined.

\smallskip If $\beta_1, \beta_2$ are two simple closed curves which both intersect
$\alpha$ essentially, then we define the \emph{subsurface distance}
\[ d_\alpha(\beta_1, \beta_2) = \mathrm{diam}(\pi_\alpha(\beta_1) \cup \pi_\alpha(\beta_2)). \]

\section{Exponential Dehn functions in genus at least $3$}
\label{sec:exponential}
\begin{theorem}\label{thm:han3-dehn}
  Let $g\geq 3$. The Dehn function of $\Han_g$ and $\Han_{g,1}$
  is at least exponential.
\end{theorem}
The core ingredient in the proof is the natural map
\[ \Han_{g,1} \to \Aut(F_g) \] induced by the action of homeomorphisms
of the handlebody $V_g$ on the fundamental group $\pi_1(V_g)
= F_g$. Our strategy will be to take a sequence of trivial words $w_n$ in
$\Aut(F_g)$ which have exponentially growing area and lift them into
the handlebody group.

\smallskip
We will spend most of this section with discussing the case of $\Han_{3,1}$ in detail; the
other cases will be derived from this special case at the end of the section. 

In \cite{BV-Dehn} the
following three automorphisms of $F_3$ are considered. Let $a, b ,c$
be a free basis of $F_3$. Then define automorphisms
\[ A:\left\{
  \begin{array}{lll}
    a & \mapsto & a\\
    b & \mapsto & b\\
    c & \mapsto & ac
  \end{array}\right., \quad
  B:\left\{
  \begin{array}{lll}
    a & \mapsto & a\\
    b & \mapsto & b\\
    c & \mapsto & cb
  \end{array}\right., \quad
  T:\left\{
  \begin{array}{lll}
    a & \mapsto & a^2b \\
    b & \mapsto & ab\\
    c & \mapsto & c
  \end{array}\right.
\]
Observe that $B$ and $T^nAT^{-n}$ commute for all $n$, and therefore we have the following 
equation
\[ T^nAT^{-n}BT^nA^{-1}T^{-n}B^{-1}  = \id\]
in the automorphism group $\Aut(F_3)$.
The crucial
result we need is the following, which is proved in \cite[Theorem~A]{BV-Dehn}.
\begin{theorem}[Bridson-Vogtmann]\label{thm:outf3-dehn}
  Consider any presentation of $\Aut(F_3)$ or $\Out(F_3)$ whose generating set
  contains the automorphisms $A, B, T$. Then, if
  $f:\NN\to\NN$ is any sub-exponential function, the loops defined by the words
  \[ w_n = T^nAT^{-n}BT^nA^{-1}T^{-n}B^{-1} \]
  cannot be filled with area less than 
  $f(n)$ for large $n$.
\end{theorem}
In particular, since the words $w_n$ have length growing linearly in
$n$, the theorem immediately implies that $\Aut(F_3)$ and $\Out(F_3)$
have exponential Dehn function.

In order to show Theorem~\ref{thm:han3-dehn}, we will realize $A, B,
T$ in a specific way as homeomorphisms of a genus $3$ handlebody. This
construction will be performed in several steps. 

\subsection*{Constructing the handlebody}

The first step is to give a specific construction of a genus $3$ handlebody $V_3$
that will be particularly useful to us. We construct $V_3$ by attaching a single
handle $H$ to an interval-bundle $V_2$ over a torus $S$ with one boundary component
(which is a genus $2$ handlebody).

\medskip To be more precise, denote by $S$ a surface of genus $1$ with
one boundary component.  We pick a basepoint $p \in \partial S$.  We
define
\[V_2 = S \times [0,1].\] This is a handlebody of genus
$2$, and its boundary $\partial V_2$ has the form
\[\partial V_2 = (S \times \{0\}) \cup (\partial S \times [0,1]) \cup 
  (S \times \{1\}). \] 
In other words, the boundary consists of two tori $S_i = S \times \{i\}, i=0,1$ and
an annulus $A= \partial S \times [0,1]$. We employ the convention that a subscript
$0$ or $1$ attached to any object in $S$ denotes its image in $S\times\{0, 1\}$.
For example, $p_0$ will denote the point $p\times\{0\}$. 

Next, we want to attach a handle in $A$ to form the genus $3$
handlebody. To this end, choose two disjoint embedded disks $D^-, D^+$
in the interior of $A$ which are disjoint from $p\times[0,1]$. Gluing
$D^-$ to $D^+$ (or, alternatively, attaching a $1$-handle at these
disks) yields our genus $3$ handlebody $V_3$. We denote by $D$ the image of the
disks $D^+, D^-$ in $V_3$.

\smallskip
Finally we will construct a core graph in $V_3$ in a way that is
compatible with our construction. Begin by choosing
two loops $a, b \subset S$ which intersect only in $p$, and which
define a free basis of $\pi_1(S, p) = F_2$. Furthermore, choose points
$q^-, q^+$ in $\partial D^-, \partial D^+$ which are identified with
each other in forming $V_3$. Then choose embedded arcs $c^+, c^-
\subset A$ from $p_1$ to $q^+, q^-$ which only intersect in $p_1$. We
denote by $c$ the loop in $V_3$ formed by traversing $c^+$ from $p_1$ to
$q^+$, then $c^-$ from $q^-$ back to $p_1$.

Then the union
\[ \Gamma = a_1 \cup b_1 \cup c \] is an embedded three-petal rose in
$\partial V_3$, so that the inclusion $\Gamma\to V_3$ induces an isomorphism on
fundamental groups (recall that $a_1 = a \times \{1\}$ and similar for
$b_1$). By slight abuse of notation, we will denote the images of the
three petals in $\pi_1(V_3)$ by $a,b,c$, and note that they form a free basis.
\begin{figure}
  \centering
  \includegraphics[width=\textwidth]{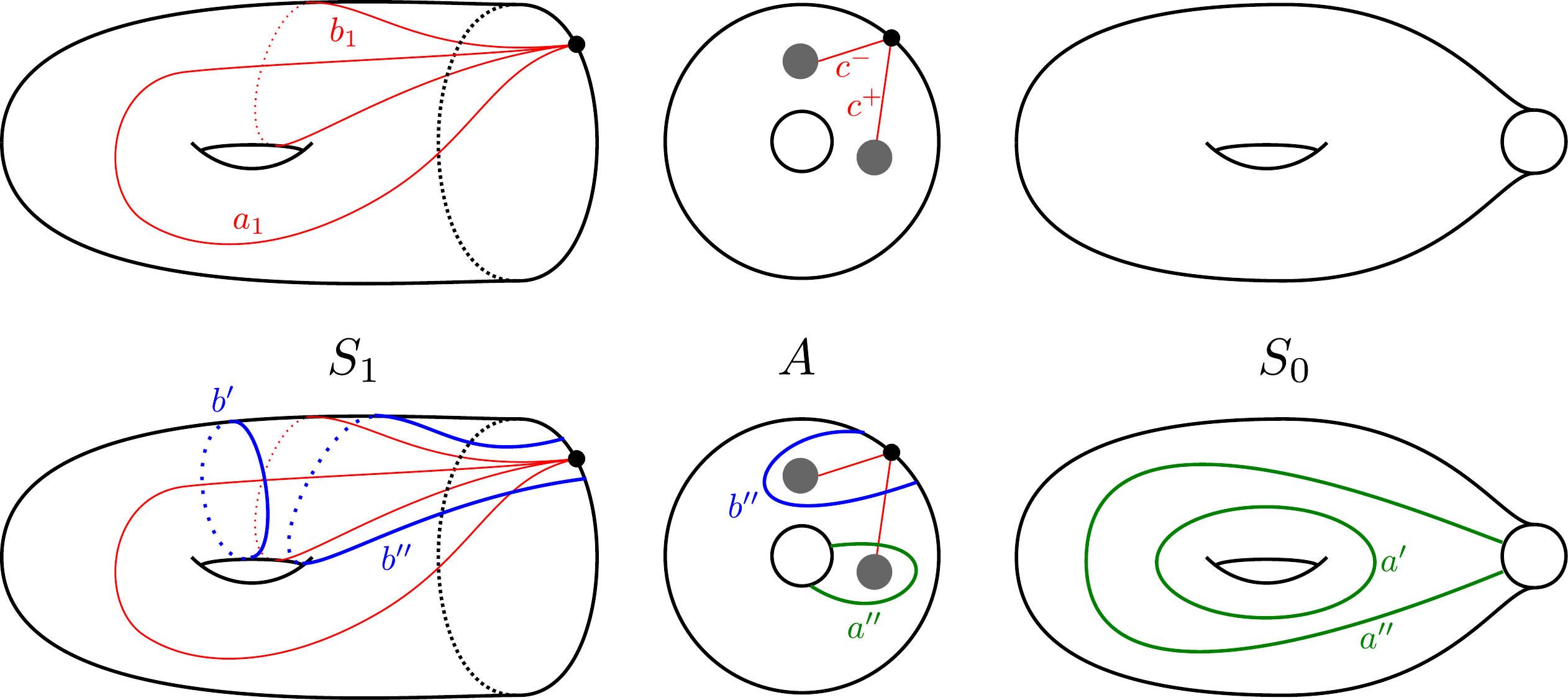}
  \caption{Constructing the handlebody and curves needed to construct the maps $\alpha, \beta, \tau$.}
  \label{fig:confusing}
\end{figure}

\subsection*{Realizing $T$ as a bundle map}
Recall that the mapping class group of a torus $S$ with one boundary
component surjects to $\Aut(F_2)$ \cite[Section~2.2.4 and
Proposition~3.19]{Primer}, and therefore there is a homeomorphism $t$
of $S$ which restricts to the identity on $\partial S$, and
so that the induced map $t:\pi_1(S,p)\to\pi_1(S,p)$ acts on the basis 
defined by the loops $a,b$ as follows:
\[ t_*(a) = a^2b, \quad t_*(b) = ab. \] The homeomorphism
$t\times\mathrm{Id}$ of $V_2$ preserves $S_0, S_1$ setwise and
restricts to the identity on $A$. By the latter fact, the
homeomorphism $t\times\mathrm{Id}$ then defines a homeomorphism $\tau$
of $V_3$, which is the identity on $A$, and in particular fixed $c$ pointwise.

We summarize the important properties of $\tau$ in the following lemma.
\begin{lemma}
  $\tau$ is a homeomorphism of $V_3$ fixing the marked point $p_1$ with
  the following properties:
  \begin{enumerate}[i)]
  \item The support of $\tau$ restricted to the boundary $\partial V_3$
    is $S_0 \cup S_1$, and it preserves both subsurfaces $S_i$ set-wise.
  \item $\tau$ acts on $\pi_1(V_3, p_1)$ in the basis $a,b,c$ as the automorphism $T$:
    \[ \tau_*(a) = a^2b, \quad \tau_*(b) = ab, \quad
    \tau_*(c) = c. \]
  \end{enumerate}
\end{lemma}

\subsection*{Realizing $A$ by a handleslide}
Intuitively, $\alpha$ will slide the end $D^+$ of the handle $H$ around 
the loop $a_0$ in the ``bottom surface'' $S_0$ of the
interval bundle $V_2 \subset V_3$.  

To be precise, let $z$ be an
arc which joins $D^+$ to $p\times\{0\}$ inside $A$, and is disjoint
from $c$. Consider a small regular neighbourhood of $\partial D^+ \cup z \cup a_0$ in
$\partial V_3$. Its boundary consists of three simple closed curves, one of which
is homotopic to $\partial D$, and the two others we denote by $a',
a''$. One of them, say $a'$, is contained in $S_0$ and will intersect
$a_0 \cup b_0$ in a single point (necessarily of $b_0$).  The other
curve $a''$ is disjoint from $S_1$, and intersects $A$ in a single
arc. Note further that $a', a''$ and $\partial D$ bound a pair of
pants in $\partial V_3$. 

Let $\alpha$ be a homeomorphism of $\partial V_3$ which defines the
product $T_{a''}T_{a'}^{-1}$ of Dehn twists in the mapping class group
of $\partial V_3$ and is supported in a small regular neighbourhood of
$a' \cup a''$.  It extends to a homemorphism of the handlebody $V_3$
by Lemma~\ref{lem:annulus-twist}, and we will denote this extension
by the same symbol.

Next, we compute the action of $\alpha$ on the fundamental group of $V_3$.
We will do this using the core graph
\[ \Gamma = a_1 \cup b_1 \cup c \] defined above. Since $a', a''$ are
disjoint from $S_1$, we have that $\alpha(a_1) = a_1, \alpha(b_1)=b_1$
for $i=1,2$. Since $a' \subset S_0$, we see that $c$ is disjoint from
$a'$. Finally, $c$ intersects $a''$ in a single point $q$. Thus,
$\alpha(c)$ is a loop, which is formed by following $c$ until the
intersection point $q$, traversing $a''$ once, and then continuing
along $c$. Observe that by pushing $a''$ first through $D$, and then
to the ``top'' half $S_1$ of the interval bundle, this loop
$\alpha(c)$ is therefore homotopic in $V_3$ (relative to $p_1$) to the
concatenation of $a_1$ and $c$.  We thus have the following properties
of $\alpha$:
\begin{lemma}\label{lem:alpha}
  $\alpha$ is a homeomorphism of $V$ fixing $p_1$ with the following properties:
  \begin{enumerate}[i)]
  \item The restriction of $\alpha$ to $\partial V$ is supported in a
    small neighbourhood of $a', a''$, where $a' \subset S_0$, and
    $a''$ is disjoint from $S_1$ and intersects $A$ in a single arc.
  \item $\alpha$ acts on $\pi_1(V_3, p_1)$ as the automorphism $A$:
    \[ \alpha_*(a) = a, \quad \alpha_*(b) = b, \quad
      \alpha_*(c) = ac. \]    
  \end{enumerate}
\end{lemma}

\subsection*{Realizing $B$ by a handleslide and a point push}
We will realise $B$ similar to $A$, by pushing $D^-$ along the loop $b$ of the
``top'' side of the interval bundle. 

To do this, we first construct an auxiliary homeomorphism $\hat{\beta}$ of
$V_3$ analogous to the previous step. Consider a regular neighbourhood of $\partial D \cup
c^- \cup b_1$. Its boundary again consists of three curves; one of which
is homotopic to $\partial D$, and we denote the others by $b',
b''$. Let $b'$ be the one which is completely contained in $S_1$ (and
thus freely homotopic to $b_1$).

As above, we can choose a homeomorphism $\hat{\beta}$ which is supported in a small
neighbourhood of $b'\cup b''$ and defines the mapping class $T_{b''}T_{b'}^{-1}$.
By Lemma~\ref{lem:annulus-twist} and the fact that $b', b''$ and $\partial D$ bound
a pair of pants, it extends to $V_3$ and we denote the extension by the same symbol.

We now compute the effect of $\hat{\beta}$ on the core graph $\Gamma$. We
begin with the petal $a_1$. It intersects both $b'$ and
$b''$ in one point each. Hence, we have that $\hat{\beta}(a_1)$ is
homotopic on $\partial V_3$, relative to the basepoint $p_1$, to the
concatenation $b_1 * a_1 * b_1^{-1}$. The loop
$b_1$ is, by construction, disjoint from $b', b''$ and so we have
$\hat{\beta}(b_1) = b_1$.  Finally, $c^+$ intersects $b''$
in a single point, and is disjoint from $b'$, while $c^-$ is
disjoint from both $b',b''$. Thus, $\hat{\beta}(c)$ is homotopic on
$\partial V_3$ to the concatenation $b_1 * c$. In total, we
see that $\hat{\beta}$ acts on our chosen basis of $\pi_1(V_3,p_1)$ as follows:
\[ \hat{\beta}_*(a) = bab^{-1}, \quad \hat{\beta}_*(b) = b, \quad
  \hat{\beta}_*(c) = bc. \] To define the homeomorphism $\beta$, we
post-compose $\hat{\beta}$ with a point push $P$ around $b_1^{-1}$ (which
is an element of the handlebody group by
Lemma~\ref{lem:point-push-facts}~i)). Since this point push has the
effect on the level of fundamental group of conjugating by $b^{-1}$
(Lemma~\ref{lem:point-push-facts}~iii)), we see that therefore $\beta$
will indeed realize the automorphism $B$ as desired.

By  Lemma~\ref{lem:point-push-facts}~ii), 
the point push homeomorphism $P$ can be chosen to be
supported in the union of $S_1$ and a small neighborhood of
$p_1$. In particular, we may
assume that the support of the point push is disjoint from the arc
$a''\cap A$ occurring in Lemma~\ref{lem:alpha}.  We summarize the required
properties of $\beta$ in the following lemma.
\begin{lemma}\label{lem:beta}
  $\beta$ is a homeomorphism of $V$ fixing the marked point $p$ with
  the following properties:
  \begin{enumerate}    
  \item The restriction of $\beta$ to $\partial V$ is the product of
    four Dehn twists about curves $d_i \subset S_0 \cup A$, all four
    of which are disjoint from the curves $a', a''$ occurring in Lemma~\ref{lem:alpha}.
  \item $\beta$ acts on $\pi_1(V_3, p_1)$ as the automorphism $B$:
    \[ \beta_*(a) = a, \quad \beta_*(b) = b, \quad
      \beta_*(c) = cb. \]
  \end{enumerate}  
\end{lemma}

\subsection*{Completing the proof}
Consider the homeomorphisms
\[ \tau^n\alpha\tau^{-n} \]
of $V_3$. Their support is contained in a small regular neighbourhood of $\tau^n(a'), \tau^n(a'')$.
Recall that $\tau$ preserves $S_0$ and hence $\tau^n(a') \subset S_0$. Thus
$\tau^n(a')$ is disjoint from all of the curves $d_i$ from
Lemma~\ref{lem:beta}. Furthermore, $\tau^n(a'') \cap A = a'' \cap A$
(since $\tau$ restricts to the identity on $A$). Hence, since the
$d_i$ from Lemma~\ref{lem:beta} are disjoint from $S_0$ and intersect
$A$ in arcs which are disjoint from $a''$, we have that $\tau^n(a'')$ is
disjoint from all $d_i$ for any $n$.

As a consequence, the 
homeomorphisms
$\tau^n\alpha\tau^{-n}$ and $\beta$ have (up to isotopy) disjoint
supports, and therefore define commuting mapping classes in $\Han_{3,1}$.
Therefore we conclude the following relation in $\Han_{3,1}$ 
\[ 1 = [\tau^n\alpha\tau^{-n}, \beta] =
\tau^n\alpha\tau^{-n}\beta\tau^n\alpha^{-1}\tau^{-n}\beta^{-1}, \]
where we have denoted the mapping classes defined by the
homeomorphisms $\alpha,\beta,\tau$ by the same symbols.  In other
words, if we choose a generating set of $\Han_{3,1}$ which contains
$\alpha, \beta, \tau$ then
\[ \omega_n =
  \tau^n\alpha\tau^{-n}\beta\tau^n\alpha^{-1}\tau^{-n}\beta^{-1} \]
are words in $\Han_{3,1}$ which define the trivial element.
Furthermore, under the map $\Han_{3,1} \to \Aut(F_3)$ they map exactly
to the words $w_n$ occurring in Theorem~\ref{thm:outf3-dehn}. By the
conclusion of that theorem, $w_n$ cannot be filled with
sub-exponential area in $\Aut(F_3)$. Since group homomorphisms
coarsely decrease area, the same is therefore true for the words
$\omega_n$. But, by construction, the length of the word $\omega_n$
grows linearly in $n$, showing that $\Han_{3,1}$ has
at least exponential Dehn function. The same is true for $\Han_3$,
since the words $w_n$ are also exponentially hard to fill in
$\Out(F_3)$ by  Theorem~\ref{thm:outf3-dehn}.

\smallskip To extend the proof of Theorem~\ref{thm:han3-dehn} to any
genus $g\geq 3$, we argue as follows. Consider the handlebody $V$ of
genus $3$ constructed above. Take a connected sum of $V$ with a
handlebody $V'$ of genus $g-3$ at a disk $D_g$ in the annulus $A$, which is 
disjoint from all curves used to define $\alpha, \beta$, to 
obtain a handlebody $V_g$ of genus $g$. The homeomorphisms
$\alpha,\beta,\tau$ can then be extended to homeomorphisms of $V_g$
which restrict to be the identity on $V'$. In this way
the words $\omega_n$ define trivial words $\hat{\omega}_n$ in $\Han_{g,1}$.

There is a natural map
\[ \iota:\Aut(F_3) \to \Aut(F_g) \] which maps an automorphism
$\varphi$ of $F\langle x_1, x_2, x_3\rangle$ to its extension to the
free group $\langle x_1, \ldots, x_g\rangle$ on $g$ generators which
fixes all $x_i, i>3$.

By construction, the words $\hat{\omega}_n$ map to the image of the
words $w_n$ under $\iota$. Corollary~4 of \cite{HM} (compare also \cite{HH-spheres}) shows that the
image of $\iota$ is a Lipschitz retract of $\Aut(F_g)$ and $\Out(F_g)$, and therefore
the words $\iota(w_n)$ are also exponentially hard to fill. By the
same argument as above, the same is therefore true for the words
$\hat{\omega}_n$.

\section{Waves in genus $2$}
\label{sec:waves}
In this section we study intersection pattern between meridians in a
genus two handlebody $V$. Recall that a \emph{cut system} of a genus two
handlebody is a pair $(\alpha_1, \alpha_2) \subset \partial V$ of
disjoint meridians with connected complement. Equivalently, cut systems
are the boundary curves of disjoint disks $D_1, D_2 \subset V$ so that
$V-(D_1\cup D_2)$ is a single $3$--ball.

The next proposition (which is true in any genus) is well-known,
see e.g. \cite[Lemma~1.1]{Masur-handlebodies} or \cite[Lemma~5.2 and the discussion preceding it]{HH1}.
\begin{proposition}\label{prop:surgery-paths}
  Suppose that $(\alpha_1, \alpha_2)$ is a cut system and $\beta$ is
  an arbitrary (multi)meridian. Either $\alpha_1\cup\alpha_2$ and
  $\beta$ are disjoint, or there is a subarc $b \subset \beta$, called
  a \emph{wave}, with the following properties:
  \begin{enumerate}[i)]
  \item The arc $b$ intersects $\alpha_1\cup\alpha_2$ only in
    its endpoints, and both endpoints lie on the same curve, say $\alpha_1$.
  \item The arc $b$ approaches $\alpha_1$ from the same side at both endpoints.
  \item Let $a, a'$ be the two components of $\alpha_1\setminus
    b$. Then exactly one of
    \[ (a\cup b, \alpha_2), (a'\cup b, \alpha_2) \] is a cut system,
    which we call \emph{the surgery defined by the wave $b$ in the
      direction of $\beta$}.
  \item The surgery defined by $b$ has fewer intersections with
    $\beta$ than $(\alpha_1, \alpha_2)$.
  \end{enumerate}
\end{proposition}
We say that a sequence $\left(\alpha_i^{(n)}\right)_n$ of cut systems is a
\emph{surgery sequence in the direction of $\beta$} if each
$\left(\alpha_i^{(n+1)}\right)$ is the surgery of $\left(\alpha_i^{(n)}\right)$ defined by
some wave $b$ of $\beta$. By Proposition~\ref{prop:surgery-paths},
these exist for any initial cut system, and they end in a system which
is disjoint from $\beta$.

\medskip The following lemmas describe certain symmetry and
uniqueness features of waves in genus $2$. They are central
ingredients in our study of projection maps in the next section.
The results of these lemma are discussed and essentially proved in
Section~4 of \cite{Masur-handlebodies}. We include a proof for
completeness and convenience of the reader.

\medskip
To describe them, it is easier to take a slightly different point of view.
Namely, let $(\alpha_1, \alpha_2)$ be a cut system of a genus $2$
handlebody $V$.
Consider $S = \partial V - (\alpha_1\cup\alpha_2)$. This is a
four-holed sphere, with boundary components
$\alpha_1^+, \alpha_1^-, \alpha_2^+, \alpha_2^-$ corresponding to the
sides of the $\alpha_i$. A wave $b$ corresponds exactly to a
subarc of $\beta$ which joins one of the boundary
components of $S$ to itself.

At this point we want to emphasize that when considering arcs in $S$, we
always consider them up to homotopy which is allowed to move the endpoints
(in $\partial S$).

\begin{lemma}\label{lem:g2wave-separate}
  Let $\beta$ be a meridian, and $b\subset S$ be a wave of $\beta$
  joining a boundary component $\alpha_i^*$ to itself. Then $b$ separates
  the two boundary components $\alpha_{1-i}^+$ and $\alpha_{1-i}^-$.
\end{lemma}
\begin{proof}
  Let $b$ be such a wave, joining without loss of generality
  $\alpha_1^+$ to itself. Suppose that $\alpha_2^-$ and $\alpha_2^+$
  are contained in the same component of $S-b$.  Then $b$ and a subarc
  $a \subset \alpha_1$ concatenate to a curve homotopic to
  $\alpha_1^-$ on $S$. Continue $b$ beyond one of its endpoints across
  $\alpha_1$ to form a larger subarc of $\beta$. It then exits from
  $\alpha_1^-$ and, by minimal position, has its next intersection
  point with $\alpha_1\cup\alpha_2$ in $a$ again. This can be
  iterated, and leads to a contradiction as the curve $\beta$ then
  cannot close up (compare
  Figure~\ref{fig:badwave}).
  \begin{figure}[h!]
    \centering
    \includegraphics[width=0.3\textwidth]{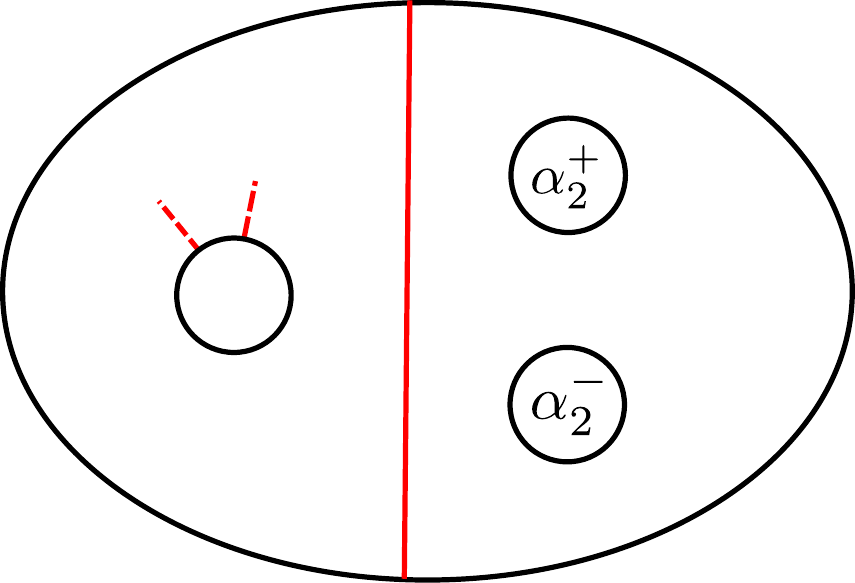}
    \caption{A "bad" wave -- the central (red) arc cannot be part of a meridian, since both ends would have to continue into the left annulus, and are unable to close up to a simple closed curve.}
    \label{fig:badwave}
  \end{figure}
\end{proof}

\begin{lemma}\label{lem:g2waves-partner}
  Suppose $(\alpha_1, \alpha_2)$ is a cut system of $V$, and $\beta$
  is any meridian. Suppose that $\beta$ has a wave $b^+$ at
  $\alpha^+_i$.
  \begin{enumerate}[i)]
  \item There is a unique essential arc $b^- \subset S$ with both
    endpoints on $\alpha^-_i$ which is disjoint from $b^+$.
  \item The arc $b^-$ from i) appears as a wave of $\beta$.
  \item Any wave of $\beta$ is homotopic to either $b^+$ or $b^-$.
  \end{enumerate}
  (The same is true with the roles of $\alpha^-_i, \alpha^+_i$ reversed)
\end{lemma}
\begin{proof}
  Let $b^+$ be a wave as in the prerequisites, joining without loss of
  generality $\alpha_1^+$ to itself. Denote the complementary
  components of $b^+$ in $S$ by $C_1$ and $C_2$.  The three boundary
  components $\alpha_1^-, \alpha_2^-, \alpha_2^+$ are contained in
  $C_1 \cup C_2$. 
  By Lemma~\ref{lem:g2wave-separate}, $\alpha_2^-$ and
  $\alpha_2^+$ are not contained in the same $C_i$. We may therefore 
  assume that $C_1$ contains $\alpha_2^-$, and $C_2$ contains $\alpha_2^+$ and 
  $\alpha_1^-$. In particular, $C_1$ is a annulus with core curve homotopic to
  $\alpha_2^-$, and $C_2$ is a pair of pants.

  Consider the boundary of a regular neighborhood of $\alpha_1^+ \cup
  b^+$ in $S$. This consists of two simple closed curves $\delta,
  \delta'$ which bound a pair of pants together with $\alpha_1^+$. Up to relabeling
  we have that $\delta \subset C_1, \delta' \subset C_2$.
  By the discussion above $\delta$ is then homotopic to
  $\alpha_2^-$, and $\delta'$ bounds a pair of pants with $\alpha_1^-, \alpha_2^+$ 
  (compare Figure~\ref{fig:goodwave}). In a pair of pants there is a
  unique isotopy class of arcs joining a given boundary component to
  itself, and hence assertion i) is true.  In an annulus there is no
  (essential) arc joining a boundary component to itself, we thus also
  conclude that there cannot be a wave of $\beta$ based at
  $\alpha_2^-$.

  \begin{figure}[h!]
    \centering
    \includegraphics[width=0.3\textwidth]{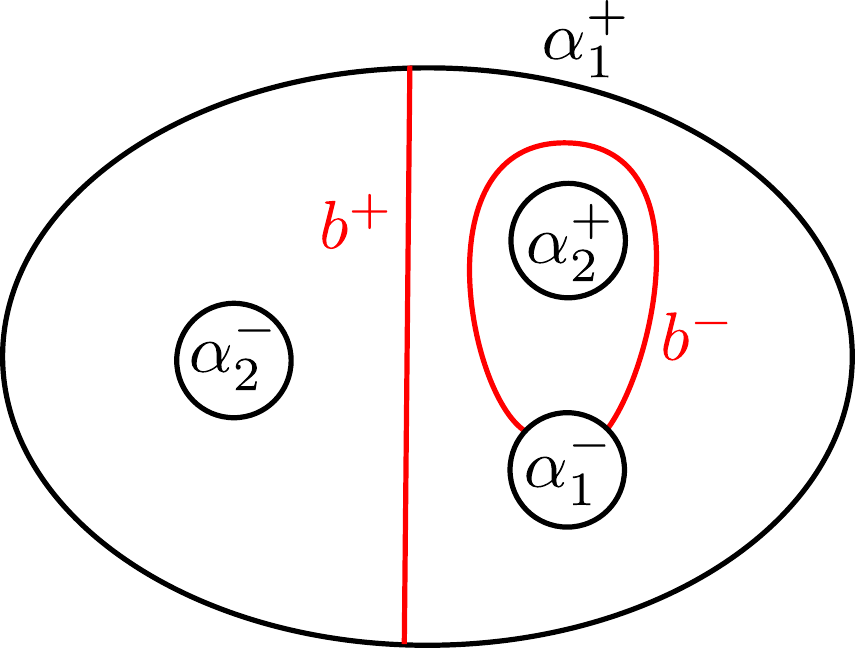}
    \caption{A "good" wave -- the central (red) arc is a wave, and it has a partner based at $\alpha_1^-$. Any other wave is necessarily homotopic to either $b^+$ or $b^-$.}
    \label{fig:goodwave}
  \end{figure}

  \smallskip Next, observe that $(\alpha_1, \alpha_2, \delta')$ is a
  pair of pants decomposition consisting of three non-separating
  curves. Let $P_1, P_2$ be the two components of $\partial V -
  (\alpha_1\cup\alpha_2\cup\delta')$. Both $P_1$ and $P_2$ are pairs of
  pants whose boundary curves are $\alpha_1, \alpha_2, \delta'$.
  Suppose that $P_1$ contains $b^+$. Since $b^+$ joins the boundary
  component $\alpha_1$ of $P_1$ to itself, and $\beta$ is embedded,
  every component of $P_1 \cap \beta$ has at least one endpoint on
  $\alpha_1$. Since $b^+$ has both endpoints on $\alpha_1^+$ we
  therefore conclude the following inequality on intersection numbers:
  \[ i(\alpha_1, \beta) > i(\alpha_2,\beta) + i(\delta',\beta). \]
  Now consider the situation in $P_2$. If there would not be an arc $b^- \subset P_2$ which
joins $\alpha_1$ to itself, then any arc in $P_2 \cap \beta$ which has one endpoint
on $\alpha_1$ has the other on $\alpha_2$ or $\delta'$. This would imply
  \[ i(\alpha_1, \beta) \leq i(\alpha_2,\beta) + i(\delta',\beta) \]
  which contradicts the inequality above. Hence, there is an arc
  $b^- \subset \beta$ which joins $\alpha_1$ to itself in $P_2$.  We have
  therefore found the desired second wave $b^-$, showing ii).

  In order to show iii), we only have to exclude a wave based at
  $\alpha_2^+$. However, this follows as above since $\alpha_2^+$ is
  contained in the annulus bounded by $b^-$ and $\alpha_2^-$.
\end{proof}

We also observe the following corollary.
\begin{corollary}\label{cor:compatible-surgeries}
  Let $(\alpha_1, \alpha_2)$ be a cut system of a genus $2$
  handlebody. Let $\beta$ be an arbitrary meridian and $b^+, b^-$ be the
  two distinct waves guaranteed by Lemma~\ref{lem:g2waves-partner}. Then
  the surgeries defined by $b^+$ and $b^-$ are equal.

  In particular, there is a unique surgery sequence starting in
  $(\alpha_1, \alpha_2)$ in the direction of $\beta$.
\end{corollary}
\begin{proof}
  The first claim is obvious from the fact that (in the notation of
  the proof of Lemma~\ref{lem:g2waves-partner}) $\delta'$ is homotopic
  to a boundary component of a regular neighborhood of
  $\alpha_1^+ \cup b^+$ and $\alpha_1^-\cup b^-$, and is therefore
  equal to the surgery defined by both $b^+$ and $b^-$. The second
  claim is immediate from the first.
\end{proof}

\section{Meridian Graphs}
\label{sec:complexes}
The purpose of this section is to construct a tree on which the
handlebody group $\Han_2$ acts, and which is crucial for the
construction of an action of $\Han_2$ on a 
${\rm CAT}(0)$ cube complex 
in Section~\ref{sec:consequences-2}.

\subsection{The wave graph (is a tree)}
\label{sec:wavegraph}
We begin with a construction of a graph which will model all possibilities
to change a cut system $Z$ to a disjoint cut system $Z'$.

To this end, we fix in this subsection once and for all a cut system $Z$. We
are interested in describing all curves $\delta$ in the complement of
$Z$ which are non-separating meridians on $\partial V$. The following
lemma allows us to encode them in a convenient way.

\begin{lemma}\label{lem:wave-description}
  Let $Z=\{\alpha_1, \alpha_2\}$ be a cut system, and $S= \partial V - Z$ its
  complementary subsurface. Fix a boundary component $\partial_0$ of
  $S$. Then the following are true:
  \begin{enumerate}[i)]
  \item A curve $\delta \subset S$ is non-separating on $\partial V$
    exactly if for both $i=1,2$, the two boundary components of $S$
    corresponding to the sides of $\alpha_i$ are contained in different
    complementary components of $\delta$.
  \item Given any $\delta$ as in i), there is a unique embedded arc
    $w$ in $S$ with both endpoints on $\partial_0$ disjoint from $\delta$,
    and it separates the two boundary components corresponding to the
    curve of $Z$ that $\partial_0$ does not correspond to. We call
    such an arc an \emph{admissible wave}.
  \item Conversely, if $w$ is any admissible wave, there is a unique curve
    $\delta$ defining it via ii).
  \item Two curves $\delta, \delta'$ as in i) intersect in two points exactly if
    the corresponding admissible wave $w, w'$ are disjoint.
  \end{enumerate}
\end{lemma}
  \begin{figure}[h!]
    \centering
    \includegraphics[width=0.8\textwidth]{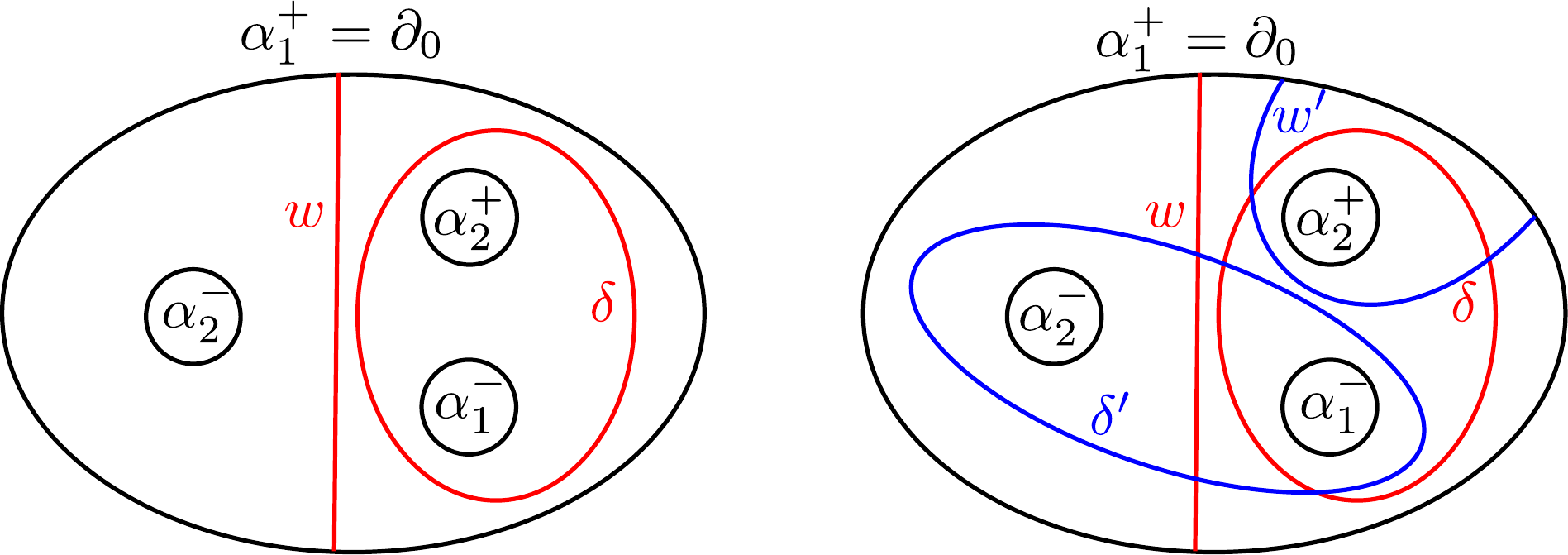}
    \caption{On the left: Passing from an admissible wave to a meridian and back as in Lemma~\ref{lem:wave-description}~ii) and~iii). On the right: Disjoint admissible waves correspond to curves intersecting twice.}
    \label{fig:wavegraphlemma}
  \end{figure}
\begin{proof}
  Assertion~i) is clear. Assertion~ii) follows because any curve
  $\delta$ as in ii) separates $S$ into two pairs of pants, and on a
  pair of pants there is a unique homotopy class of embedded arcs with
  endpoints on a specified boundary component. Assertion~iii) follows
  since such an arc cuts $S$ into an annulus and a pair of pants
  (compare the proof of Lemma~\ref{lem:g2waves-partner}). It remains
  to show the final Assertion~iv). First observe that if $w, w'$ are
  disjoint then the corresponding curves constructed in iii) indeed
  intersect in two points.

  Finally, suppose that $\delta$ is a curve defined by an arc $w$ as in iii).
  The arc $w$ separates $S$ into $S_1$ and $S_2$. Without loss of generality
  we may assume that $S_1$ is an annulus containing $\alpha_2^-$ and $S_2$ is a
  pair of pants.

  Suppose now that $\delta'$ intersects $\delta$ in two points. Then
  $\delta'$ will also intersect $w$ in two points, and hence there is
  a subarc $d' \subset \delta'$ in $S_2$ with both endpoints on
  $w$. Since $S_2$ is a pair of pants, there is a unique arc $w'$ in
  $S_2$ with endpoints on $\partial_0$ which is disjoint from $d'$, and
  therefore from $\delta'$. By the uniqueness of ii) this is the arc
  defining $\delta'$, and since it is contained in $S_1$ with endpoints
  on $\partial_0$ it is indeed disjoint from $w$.
\end{proof}
Motivated by this lemma, we make the following definition.
\begin{definition}
  The \emph{wave graph $\W(Z)$ of $Z$} is the graph whose vertices
  correspond to admissible waves $w$ based at $\partial_0$, and whose
  edges correspond to disjointness.
\end{definition}
For future reference, we record the following immediate corollary of
Lemma~\ref{lem:wave-description}:
\begin{corollary}\label{cor:wave-graph}
  The wave graph $\W(Z)$ is isomorphic to the graph whose vertices
  correspond to non-separating meridians $\delta$ which are disjoint from $Z$,
  and where vertices corresponding to $\delta, \delta'$ are connected
  by an edge if $i(\delta, \delta') = 2$.
\end{corollary}
\begin{remark}
  As a consequence of Corollary~\ref{cor:wave-graph}, the wave graph
  $\W(Z)$ can be identified with a subgraph of the Farey graph of the
  four-holed sphere $\Sigma_{0,4}$. To describe this subgraph, recall
  that every edge of the Farey graph is contained in two
  triangles. The three vertices of such a triangle correspond to the
  three different ways to separate the four punctures of
  $\Sigma_{0,4}$ into two sets. The condition for the meridian to be
  non-separating excludes one of these -- so the wave graph is obtained
  from the Farey graph of $\Sigma_{0,4}$ by removing one vertex and
  two edges from each triangle. This point of view can be used to show
  that the wave graph is a tree (Theorem~\ref{thm:wave-tree}), but we
  will prove this theorem using different techniques which are useful later.
\end{remark}
We also record the following
\begin{lemma}\label{lem:wave-connected}
  The wave graph $\W(Z)$ is connected.
\end{lemma}
\begin{proof}
  This is a standard surgery argument, using induction on intersection
  number. Suppose $w, w'$ are any two admissible waves corresponding
  to vertices in $\W(Z)$. If they are not disjoint, consider an
  initial segment $w_0 \subset w$ which intersects $w'$ only in its
  endpoint $\{q\} = w_0 \cap w'$. Let $w'_1, w'_2$ be the two
  components of $w'\setminus\{q\}$.  Then both $x_i=w_0 \cup w'_i$ are
  arcs which are disjoint from $w'$ and have smaller intersection with
  $w$. It is easy to see that exactly one of them is admissible
  (compare also Figure~\ref{fig:wave-projection}), and hence $w'$ is
  connected to a vertex $x_i \in \W(Z)$ which corresponds to an arc of
  strictly smaller intersection number with $w$. By induction, the
  lemma follows.
\end{proof}
In order to study $\W(Z)$ further, we define the following projection
maps. 

First, given a vertex $w \in \W(Z)$ corresponding to an admissible
wave (which we denote by the same symbol), we define the set
$\A(w)$ to consist of embedded arcs $a$
in $S-w$ with one endpoint on $w$ and the second endpoint on $\partial_0$,
which are not homotopic into $w \cup \partial_0$ (up to homotopy of
such arcs).

We observe that any arc $a$ corresponding to a vertex in $\A(w)$ cuts
the pair of pants $S-w$ into two annuli. From this, we obtain the
following consequence (see also Figure~\ref{fig:wave-projection}):
\begin{lemma}\label{lem:waveprojection-unique}
  \begin{enumerate}[i)]
  \item No two  arcs $a, a'$ corresponding to different vertices in $\A(w)$
    are disjoint.
  \item Given an arc $a$ corresponding to a vertex in $\A(w)$ there is a
    unique admissible wave $w'$ which is disjoint from $w$ and from $a$.
  \item No two distinct admissible waves are disjoint from $w$.
  \end{enumerate}
\end{lemma}
\begin{proof}
  If $a'$ is an arc with endpoint on $w$ and $\partial_0$, which is
  disjoint from $a$, then (since both components of $S-(w\cup a)$ are
  annuli) it is either homotopic into $\partial_0 \cup w$, or
  homotopic to $a$. This shows i). To see the second
  claim, note that there are two homotopy classes of arcs in $S$ with
  endpoints in $\partial_0$ disjoint from $w$. Exactly one of them is
  an admissible wave, showing ii). For the third claim, consider
  an admissible wave $w'$ disjoint from $w$.  Since it is distinct
  from $w$, it separates the two boundary components of $S$ which
  are contained in $S-w$. If $w''$ is now any other arc with endpoints
  on $\partial_0$ which is disjoint from $w$ and $w'$, then it is either
  homotopic to one of them, or not admissible, showing iii).
\end{proof}
Given a vertex $w \in \W(Z)$, we now define a map
\[ \pi_w: \W(Z) \to \A(w) \]
in the following way:
\begin{enumerate}[i)]
\item If $x \in \W(Z)$ is not disjoint from $w$, consider an initial
  segment $x_0 \subset x$ one of whose endpoints is on $\partial_0$,
  the other is on $w$, and so that its interior is disjoint from $w$.
  We then put $\pi_w(x) = x_0$ (see Figure~\ref{fig:wave-projection}).
  This is well-defined by Lemma~\ref{lem:waveprojection-unique}~i).
\item If $y$ is disjoint from $w$, we let $\pi_w(y)$ be the an 
  arc in $\A(w)$, which is disjoint from $y$.
  This is well-defined by Lemma~\ref{lem:waveprojection-unique}~ii).
\end{enumerate}
\begin{figure}[h!]
  \centering
  \includegraphics[width=0.4\textwidth]{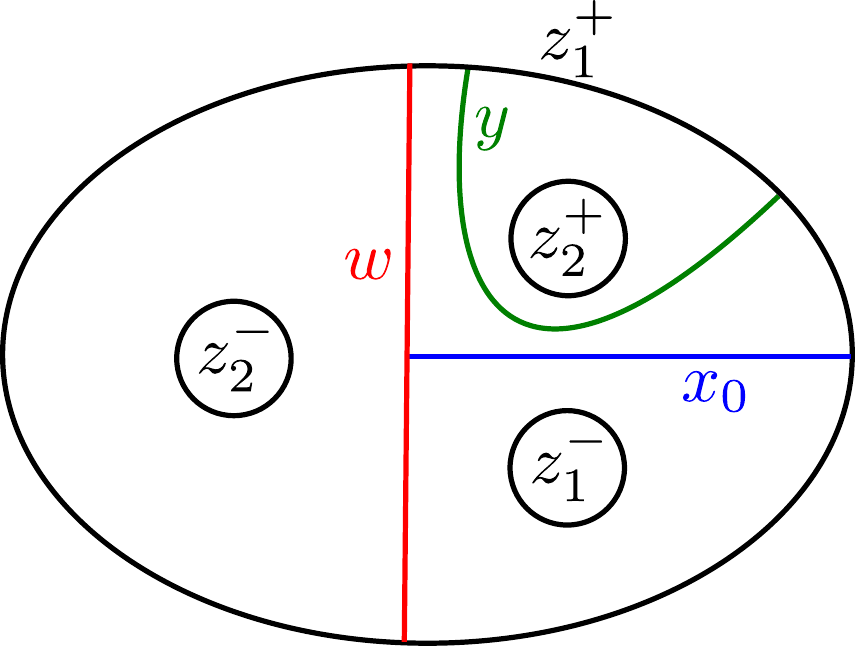}
  \caption{Projecting a wave into $\A(w)$. An arbitrary wave $x$
    intersecting $w$ has an initial segment $x_0$ joining $\partial_0$
    to $w$, which is its projection. For a disjoint wave $y$ there is
    a unique arc as above disjoint from it.}
  \label{fig:wave-projection}
\end{figure}
We then get the following consequences.
\begin{proposition}\label{prop:wave-outside-link}
  If $x,y \in W(Z)$ are both not disjoint from $w$, but $x,y$ are
  disjoint from each other, then $\pi_w(x) = \pi_w(y)$.
\end{proposition}
\begin{proof}
  This is immediate from Lemma~\ref{lem:waveprojection-unique}~i).
\end{proof}
\begin{proposition}\label{prop:wave-entering-link}
  If $x\in \W(Z)$ is not disjoint from $w$, and $y\in \W(Z)$ is
  disjoint from $x$ and from $w$, then $\pi_w(x) = \pi_w(y)$.
\end{proposition}
\begin{proof}
  This is immediate from Lemma~\ref{lem:waveprojection-unique}~ii).
\end{proof}
We are now ready to prove the following central result. 
\begin{theorem}\label{thm:wave-tree}
  For any cut system $Z$, the wave graph $\W(Z)$ is a tree.
\end{theorem}
\begin{proof}
  We have already shown that $\W(Z)$ is connected
  (Lemma~\ref{lem:wave-connected}), and so it suffices to show that
  there are no embedded cycles. As a first step, note that by
  Lemma~\ref{lem:waveprojection-unique} no two vertices in the link of $w$
  are joined by an edge. Namely, if $w_1, w_2$ are in the link of $w$ (i.e.
  disjoint from $w$), Lemma~\ref{lem:waveprojection-unique}~iii) states
  that they cannot be disjoint from each other. This in particular implies
  that there are no cycles of length $\leq 3$.
 
  Hence, suppose that $w_0,w_1 \ldots, w_n$ is a cycle of length
  $n\geq 4$. By possibly passing to a sub-cycle, we may assume that
  \begin{enumerate}
  \item $w_1, w_n$ are the only vertices in the link of $w_0$ (i.e. the only
    arcs disjoint from $w_0$).
  \item The arcs $w_1, w_n$ are distinct.
  \end{enumerate}
  Consider now $\pi_{w_0}(w_1) = a \in \A(w_0)$. Applying
  Proposition~\ref{prop:wave-entering-link} (for $x=w_2, y=w_1$) we
  obtain that $\pi_{w_0}(w_2) = a$ as well. Inductively applying
  Proposition~\ref{prop:wave-outside-link}, we obtain that
  $\pi_{w_0}(w_i) = a$ for any $i=1, \ldots, n-1$. Finally, applying
  Proposition~\ref{prop:wave-entering-link} again (for $x=w_{n-1}, y=w_n$) we
  see that $\pi_{w_0}(w_n) = a = \pi_{w_0}(w_1)$.

  However, $w_n$ and $w_1$ are assumed to be distinct, and therefore cannot
  have the same projection by Lemma~\ref{lem:waveprojection-unique}~ii).
\end{proof}

\subsection{The non-separating meridional pants graph (is a tree)}
\label{sec:cutpants}
We now come to the central object we will use to study $\Han_2$.
\begin{definition}
  \begin{enumerate}[i)]
  \item The \emph{non-separating meridional pants graph} $\CP_2$ has vertices corresponding
    to pants decompositions $X=\{ \delta_1, \delta_2, \delta_3 \}$ so
    that all $\delta_i$ are non-separating meridians. We put an edge
    between $X$ and $X'$ if they intersect minimally, i.e. in two
    points.
  \item For any cut system $Z$, let $\CP_2(Z)$ be the full subgraph
    corresponding to all those pants decompositions
    $X=\{ \delta_1, \delta_2, \delta_3 \}$ which contain $Z$.
  \end{enumerate}
\end{definition}
From Corollary~\ref{cor:wave-graph} and Theorem~\ref{thm:wave-tree}
we immediately obtain
\begin{corollary}\label{cor:cpz-tree}
  For any cut system $Z$ the subgraph $\CP_2(Z)$ is a tree.
  Any two such subtrees intersect in at most a single point.
\end{corollary}
We will use these subtrees in order to study $\CP_2$. We begin with the
following.
\begin{lemma}\label{lem:cp2-connected}
  The graph $\CP_2$ is connected.
\end{lemma}
\begin{proof}
  Let $X, Y$ be pants decompositions corresponding to vertices of
  $\CP_2$. We will construct a path joining $X$ to $Y$ in $\CP_2$.
  Choose two curves $\{ \delta_1, \delta_2 \} = Z_1$ from $X$ -- these
  will form a cut system by the definition of $\CP_2$. Now, consider
  the surgery sequence $(Z_i)$ starting in $Z_1$ in the direction of
  $Y$. Let $n$ be so that $Z_n$ is disjoint from $Y$. As $Y$ is a
  pants decomposition, this implies that actually $Z_n \subset Y$.

Also, by definition of surgery sequences, for any $i$ the cut systems $Z_{i-1}$
and $Z_{i+1}$ are contained in the complement of the cut system $Z_i$,
and thus $Z_i \cup Z_{i-1}$ and $Z_i\cup Z_{i+1}$ correspond to
vertices in the tree $\CP_2(Z_i)$. Hence, these vertices can be joined by a
path. The desired path $(X_i)$ will now
be obtained by concatenating all of these paths. To be more precise,
we will have
\begin{enumerate}
\item There are numbers $1=m(0), m(1), \ldots, m(n)$ so that for all
  $m(i-1)< j \leq m(i)$ the pants decomposition $X_j$ contains $Z_i$.
\item For all $m(i-1)< j \leq m(i)$ the pants decomposition $X_j$ are
  a geodesic in $\CP_2(Z_i)$.
\end{enumerate}
From the description above it is clear that these sequences exist, showing
Lemma~\ref{lem:cp2-connected}.
\end{proof}
We will now define projections of $\CP_2$ onto the subtrees
$\CP_2(Z)$. To this end, let $Z$ be a cut system. We define a
projection 
\[ \pi_Z: \CP_2 \to \CP_2(Z) \]
in the following way.
\begin{enumerate}[i)]
\item If $X$ is disjoint from $Z$, we simply put $\pi_Z(X) = X$.
\item If $X$ intersects $Z$, then there is a wave $w$ of $X$ with
  respect to $Z$, and we define $\pi_Z(X) = Z \cup \{\delta\}$, where
  $\delta$ is the surgery defined by the wave
  $w$. Corollary~\ref{cor:compatible-surgeries} implies that this is
  well-defined.
\end{enumerate}
\begin{proposition}\label{prop:projection-far}
  Suppose that $X, Y \in \CP_2$ are connected by an edge, and assume
  that both $X, Y$ are not disjoint from $Z$. Then
  $\pi_Z(X) = \pi_Z(Y)$.
\end{proposition}
\begin{proof}
  Since $X$ and $Y$ are not disjoint from $Z$, there are 
  waves $w_X, w'_X, w_Y, w'_Y$ as in
  Lemma~\ref{lem:g2waves-partner}. We claim that unless
  $\{w_X, w'_X\} = \{w_Y, w'_Y\}$, the total number of intersections between
  $\{w_X, w'_X\}, \{w_Y, w'_Y\}$ is at least $4$, contradicting that $X, Y$
  are joined by an edge.

  However, this is seen in a similar way as the argument in
  Lemma~\ref{lem:g2waves-partner} in different cases (compare Figure~\ref{fig:projection-stable}). First observe that as the waves are
  arcs in a four-holed sphere joining the same boundary to itself,
  two waves are either disjoint or intersect at least in two points.

  Suppose first that the waves of $Y$ are based at the same component
  of $Z$ as the ones of $X$, and assume that $w_X, w_Y$ approach from
  the same side. If $w_X$ and $w_Y$ are disjoint, then by the
  uniqueness statement of Lemma~\ref{lem:g2waves-partner} we have that
  $\{w_X, w'_X\} = \{w_Y, w'_Y\}$, and thus the claim.
  If $w_X, w_Y$ are not disjoint, then $w_Y$ also intersects
  $w'_X$ (in at least two points), and we are done.

  The case where the waves of $X$ and $Y$ are based at different components
  is similar, noting that each of $w_Y, w_Y'$ needs to intersect at least one
  of the $w_X, w'_X$.
 \begin{figure}[h!]
    \centering
    \includegraphics[width=\textwidth]{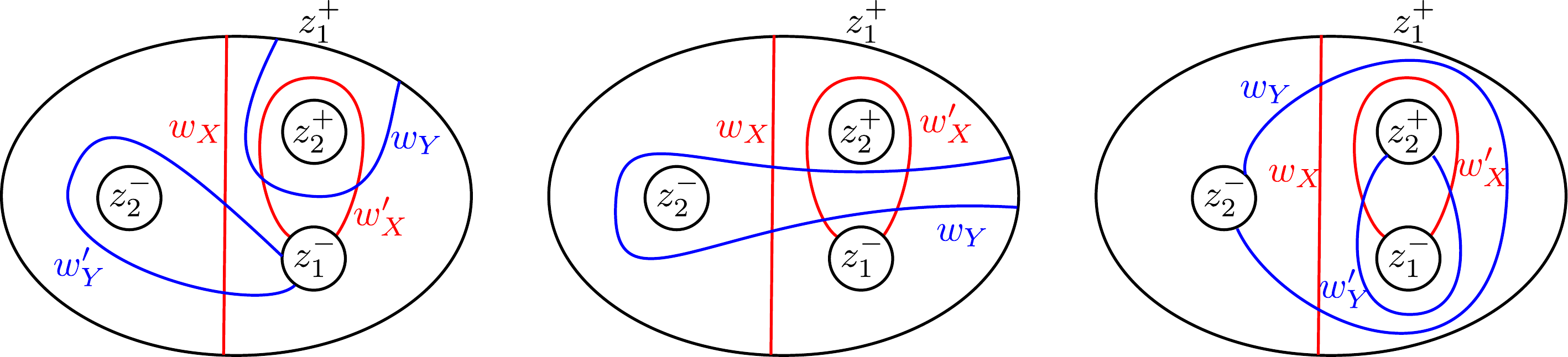}
    \caption{The three cases in the proof of Proposition~\ref{prop:projection-far}. In any configuration, different waves generate at least four intersection points.}
    \label{fig:projection-stable}
  \end{figure}
\end{proof}

\begin{proposition}\label{prop:projection-into-link}
  Suppose that $X, Y \in \CP_2$ are joined by an edge, that $X$ is not
  disjoint from $Z$, but $Y$ is disjoint from $Z$.
  Then $\pi_Z(X) = \pi_Z(Y)$.
\end{proposition}
\begin{proof}
  Since $X$ is not disjoint from $Z$, it has a pair of waves
  $w_X, w_X'$ as in Lemma~\ref{lem:g2waves-partner}. $Y$ differs from
  $X$ by exchanging a single curve of $X$. Since $Y$ is disjoint from
  $Z$ but $X$ is not, two curves $x_1, x_2$ of $X$ are disjoint from
  $Z$, while a third one $x_3$ contributes the waves.  The pair of
  curves $\{x_1, x_2\}$ which is disjoint from $Z$ has to be distinct
  from $Z$ as otherwise there could not be any waves. Hence, $X$ and $Z$
  have precisely one curve in common, say $x_1$. The other curve $x_2$, being
  disjoint from one of the curves in $Z$ and the waves, is then
  necessarily the surgery along that wave (see
  Figure~\ref{fig:becoming-disjoint}).
  
  The move from $X$ to $Y$ replaces the curve $x_3$ contributing the
  waves, and therefore keeps $x_2$ -- which will be the projection of
  both $X$ and $Y$.
\end{proof}
  \begin{figure}[h!]
    \centering
    \includegraphics[width=0.4\textwidth]{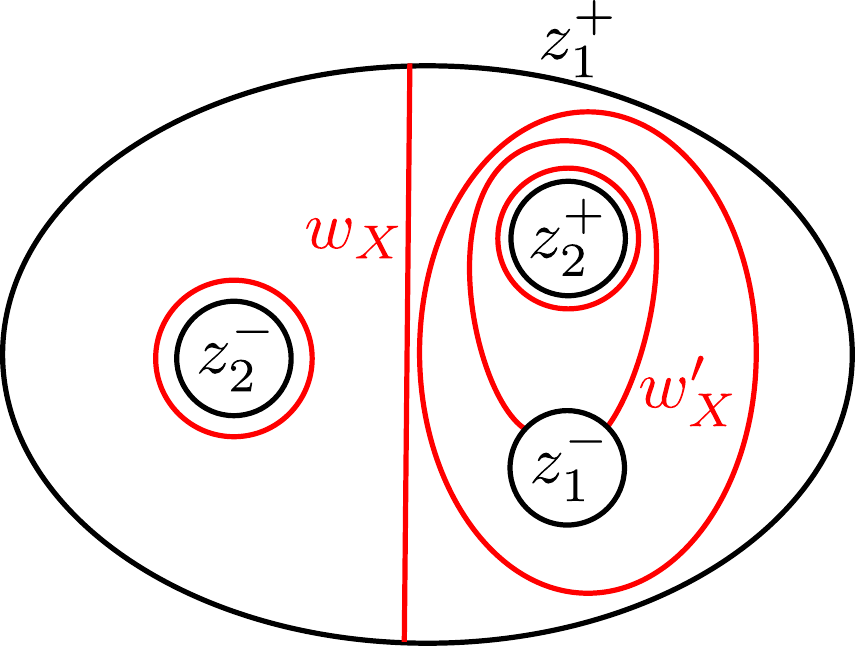}
    \caption{The situation in
      Proposition~\ref{prop:projection-into-link}. $X$ has two waves
      $w_X, w'_X$ contributed by $x_3 \in X$ and one curve $x_1=z_2$ in common. The remaining curve $x_2\in X$ in
      $X$ is then necessarily the surgery at the wave. Doing a single
      move to make $X$ disjoint from $Z$ will replace the curve $x_3$
      contributing the wave, keeping $x_2$ and therefore the projection fixed.}
    \label{fig:becoming-disjoint}
  \end{figure}

Together, these propositions can be rephrased as saying
that the projection $\pi_Z$ from $\CP_2$ to $\CP_2(Z)$ can only change
along a path while that path is actually contained within $\CP_2(Z)$.

%
\begin{theorem}\label{thm:CP-tree}
  The graph $\CP_2$ is a tree. 
\end{theorem}
\begin{proof}
  The proof is very similar to the proof of
  Theorem~\ref{thm:wave-tree}. We have already seen connectivity of
  $\CP_2$ in Lemma~\ref{lem:cp2-connected}. Suppose that $P_1, \ldots, P_n$
  is a nontrivial cycle in $\CP_2$. We choose a cut system $Z \subset P_1$.
  Since $\CP_2(Z)$ is a tree, the cycle cannot be completely contained in
  $\CP_2(Z)$. Hence, by passing to a sub-cycle we may assume
  \begin{enumerate}[i)]
  \item There is a number $k$, so that $P_i$ is a vertex of $\CP_2(Z)$ exactly
    for $1\leq i\leq k$.
  \item The vertices $P_1, P_k$ are distinct.
  \end{enumerate}
  Applying Proposition~\ref{prop:projection-into-link} (with
  $Y=P_k$ and $X=P_{k+1}$) we conclude that $\pi_Z(P_{k+1}) =
  \pi_Z(P_k)$. Inductively applying
  Proposition~\ref{prop:projection-far} we conclude that
  $\pi_Z(P_i) = \pi_Z(P_k)$ for $i\leq n$. Applying
  Proposition~\ref{prop:projection-into-link} again (for
  $X=P_n, Y=P_1$), we conclude that $\pi_Z(P_1) = \pi_Z(P_K)$. But,
  since $P_1, P_n \in \CP_2(Z)$, we conclude $P_1 = P_k$, violating
  assumption ii) above. This shows that $\CP_2$ admits no cycles and
  therefore is a tree.
\end{proof}

\subsection{Controlling Twists}
\label{subsec:twists}
In this subsection we study how subsurface projections to annuli
around non-separating meridians $\alpha$ behave along geodesics in $\CP_2$.
We begin with the following lemma, which is likely known to experts.
\begin{lemma}\label{lem:coarsely-transitive-projection}
  Suppose that $Y \subset \partial V$ is a subsurface, and that
  $\alpha \subset Y$ is an essential simple closed curve. Suppose that
  $\beta_1, \beta_2$ are two curves which intersect $\partial Y$, and suppose
  further that there is an arc $b$ in $Y$ which intersects $\alpha$ and so
  that there are subarcs $b_i \subset Y\cap \beta_i$ which
  are isotopic to $b$ . Then
  \[ d_\alpha(\beta_1, \beta_2) \leq 5 \]
  (here, the subsurface distance $d_\alpha$ is seen as curves on $S$, not $Y$).
\end{lemma}
\begin{proof}
  Up to isotopy we may assume that the curves $\beta_i$ both actually
  contain $b$ (and are in minimal position with respect to themselves,
  $\alpha$ and $\partial Y$).
  
  Let $S_\alpha \to \partial V$ be the annular cover corresponding to
  $\alpha$, and let $\hat{\alpha}$ be the unique closed lift of
  $\alpha$. Suppose that $b$ joins components $\delta_1, \delta_2$ of
  $\partial Y$.

  Consider a lift $\hat{b}$ of $b$ which intersects
  $\hat{\alpha}$. Its endpoints are contained in lifts
  $\hat{\delta}_i$ of the curves $\delta_i$.  Observe that both the
  lifts $\hat{\delta}_i$ ($i=1,2$) do not connect different boundary
  components of the annulus $S_\alpha$ (as the curves $\delta_i$ are
  disjoint from $\alpha$), and therefore $\hat{\delta}_i$
  bounds a 
  disk $D_i\subset S_\alpha$ whose closure in the closed annulus
  $\overline{S_\alpha}$ intersects the boundary of $\overline{S_\alpha}$
  in a connected subarc.

  Now consider lifts $\hat{\beta}_j$ which contain the arc $\hat{b}$.
  These are concatenations of an arc in $D_1$, the arc $\hat{b}$, and
  an arc in $D_2$. As the arcs in $D_i$ can intersect in at most one
  point (otherwise, minimal position of $\beta_1, \beta_2$ would be
  violated!), this implies that there are two lifts of $\beta_i$ which
  intersect in at most $2$ points. This shows the lemma.
\end{proof}

We can use this lemma to prove the following result how subsurface
projections $\pi_\alpha$ into annuli around meridians change along
$\CP_2$--geodesics if these geodesics never involve the curve $\alpha$.
This should be seen as the direct analog of the bounded geodesic projection theorem
and its variants for hierarchies which are developed in \cite{MM2}.
\begin{proposition}\label{prop:annulus-projection-constant}
  Suppose that $X_i$ is a geodesic in $\CP_2$, and that $\alpha$ is a
  non-separating meridian. Suppose none of the pants decompositions $X_i$
  contains 
  $\alpha$. Then the subsurface projection
  $\pi_\alpha(X_i)$ is coarsely constant along $X_i$: there is a
  constant $K$ independent of $\alpha$ and the sequence, so that
  \[ d_\alpha(X_i, X_j) \leq K.\]
\end{proposition}
\begin{proof}
  First observe that as none of the $X_i$ contains $\alpha$, actually all $X_i$
  intersect $\alpha$.
  Let $Z$ be a cut system completing $\alpha$ to a pants decomposition.
  Since $X_1$ intersects $\alpha$, we may assume that there
  is a wave $w$ of $X_1$ which intersects $\alpha$.

  Now consider the path $X_i$. Observe that since $\CP_2(Z)$ is a
  subtree of $\CP_2$, the intersection of the path $X_i$ with
  $\CP_2(Z)$ is a path, say $X_i, k\leq i\leq k'$.

  For $i=1, \ldots, k-1$, we have $X_i \notin \CP_2(Z)$, and by 
  Proposition~\ref{prop:projection-far} the pants decompositions $X_i$
  will therefore all have the same wave $w$. By
  Lemma~\ref{lem:coarsely-transitive-projection} this implies that the
  subsurface projection $\pi_\alpha$ is coarsely constant for
  $X_1, \ldots, X_{k-1}$ where $k$ is the first index so that
  $X_k \in \CP_2(Z)$.  The projections of $X_{k-1}$ and $X_{k}$ are
  uniformly close since $X_{k-1}, X_k$ are disjoint and both intersect
  $\alpha$.  Similarly, the projections of $X_{k'}, X_{k'}$ are
  uniformly close, and applying Proposition~\ref{prop:projection-far}
  and Lemma~\ref{lem:coarsely-transitive-projection}, we see that the
  projection is coarsely constant for $i\geq k'$.

  Hence, it suffices to show the statement of the proposition for
  paths which are completely contained in $\CP_2(Z)$.

  So, consider a path $X_j$ in $\CP_2(Z)$ which is never disjoint from
  a curve $\alpha \subset \partial V - Z$. Let $w$ be the wave
  corresponding to $\alpha$. Then, the projection $\pi_w(X_i)$ is
  constant by Proposition~\ref{prop:projection-far}, and therefore the projection
  $\pi_\alpha(X_i)$ is coarsely constant, arguing as in
  Lemma~\ref{lem:coarsely-transitive-projection}.
\end{proof}
Finally, we study the case where $\alpha$ does appear as one of the curves
along a $\CP_2$--geodesic.
\begin{corollary}\label{cor:twist-control}
  Let $\alpha$ be a non-separating meridian, and $X_i$ be a geodesic in
  $\CP_2$, which does become disjoint from $\alpha$. Then there are
  $i_0 \leq i_1$ so that the following holds:
  \begin{enumerate}[i)]
  \item For $i<i_0$ the subsurface projection $\pi_\alpha(X_i)$ is
    coarsely constant.
  \item For $i_0\leq i \leq i_1$, the curve $\alpha$ is contained in $X_i$.
  \item For $i>i_0$ the subsurface projection $\pi_\alpha(X_i)$ is
    coarsely constant.
  \end{enumerate}
\end{corollary}
\begin{proof}
  In light of the previous
  Proposition~\ref{prop:annulus-projection-constant} the only thing
  that remains to show is that an interval $i_0 \leq i_1$ exists with
  the property that $X_i$ contains $\alpha$ exactly for $i_0 \leq i
  \leq i_1$. This follows since the set $\CP_2(\alpha)$ of
  non-separating meridional pants decompositions containing $\alpha$ is
  the union of $\CP_2(Z)$ for $Z$ a cut system containing $\alpha$,
  which is a connected union of subtrees, hence itself a
  subtree. Therefore, a geodesic in $\CP_2$ intersects $\CP_2(\alpha)$
  in a path.
\end{proof}

\section{A geometric model for $\Han_2$}
\label{sec:consequences-2}
In this section we define a geometric model for the handlebody group
(which is very similar to the one employed in \cite{Transactions}) and use
the results from Section~\ref{sec:complexes} in order to study the geometry
of the genus $2$ handlebody group. A first step is the following lemma.
\begin{lemma}\label{lem:dualsystem}
  Suppose that $X \in \CP_2$ is a pants decomposition, and $X =
  \{\delta_1, \delta_2, \delta_3\}$. Given $i \in \{1,2,3\}$ there is
  a curve $\beta_i$ with the following properties:
  \begin{enumerate}[i)]
  \item $\beta_i$ is a non-separating meridian.
  \item $\delta_i$ and $\beta_j$ are disjoint for $i \neq j$.
  \item $\delta_i$ and $\beta_i$ intersect in exactly two points.
  \end{enumerate}
  Furthermore, the curve $\beta_i$ is uniquely defined by these properties up to
  Dehn twist about the curve $\delta_i$.
\end{lemma}
\begin{proof}
Assume without loss of generality that $i=3$.
  Consider the surface $S$ obtained by cutting $\partial V$ at
  $\delta_1, \delta_2$ as in Section~\ref{sec:waves}. The curve
  $\delta_3$ defines an admissible wave $w$ as in
  Lemma~\ref{lem:wave-description}, and by the same lemma any curve
  $\beta_3$ with the desired properties will be defined by an
  admissible wave $w'$ which is disjoint from $w$. Arguing as in
  Lemma~\ref{lem:waveprojection-unique}, such an admissible wave $w'$
  exists and is unique up to Dehn twist in $\delta_3$. This shows both
  claims of the lemma.
\end{proof}
\begin{definition}
  If $X \in \CP_2$, we call a curve $\beta_i$ \emph{dual to $\gamma_i
    \in X$} if it satisfies the conclusion of
  Lemma~\ref{lem:dualsystem}. A set $\Delta = \{\beta_1, \beta_2,
  \beta_3\}$ containing a dual to each $\gamma_i \in X$ is called a
  \emph{dual system} to $X$.
\end{definition}
Since the handlebody group acts transitively on pants decompositions
consisting of non-separating meridians, we see that the handlebody
group also acts transitively on pairs $(X, \Delta)$ where $X \in \CP_2$ and
$\Delta$ is a dual system to $X$.

Next, we will describe a procedure to canonically modify the dual
system when changing the pants decomposition $X$ to an adjacent one
$X'$ in $\CP_2$. Suppose that $X \in \CP_2$ is a pants
decomposition, and that $\Delta$ is a dual system. Suppose that
$\delta \in \Delta$ is the dual curve to a curve $\gamma \in X$. Then
the system $X' = X\cup\{\delta\}\setminus\{\gamma\}$ obtained by
swapping $\gamma$ for $\delta$ is also a pants decomposition
consisting of non-separating meridians, and defines a vertex $X'$
adjacent to $X$ in $\CP_2$.
We say that $X'$ is obtained from $(X, \Delta)$ by
\emph{switching $\gamma$}, 

The curve $\gamma$ is dual to $\delta$ in $X'$ in the sense of
Lemma~\ref{lem:dualsystem}. However, the other curves $\delta_1,
\delta_2$ of $\Delta\setminus\{\delta\}$ are not -- each of them will
intersect $\delta$ in four points. The following lemma will allow us
to clean the situation up in a unique way.
\begin{lemma}\label{lem:cleanup}
  Suppose that $X$ is a pants decomposition, and $\Delta$ is a dual
  system.  Let $\gamma \in X$ be given, and suppose $X'$ is obtained from $(X,
  \Delta)$ by switching $\gamma$. 

  Let $\gamma'\in X$ be distinct from $\gamma$, and $\delta'\in
  \Delta$ its dual. Then there is a dual curve $c(\delta')$ to $\gamma'$
  for the system $X'$, and the assignment $\delta' \mapsto c(\delta')$
  commutes with Dehn twists about $\gamma'$.
\end{lemma}
\begin{figure}
  \centering
  \includegraphics[width=\textwidth]{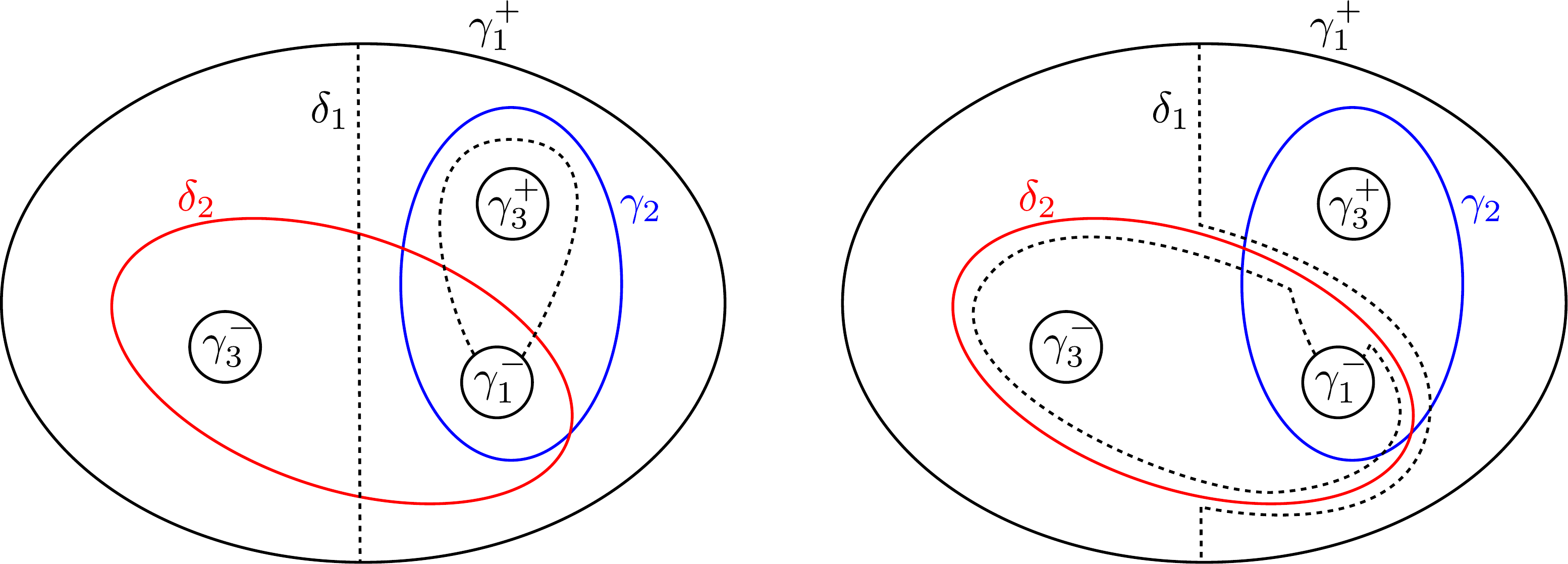}
  \caption{The cleanup move for dual curves in a switch.}
  \label{fig:cleanup-move}
\end{figure}
\begin{proof}
  Suppose that $X = \{\gamma_1, \gamma_2, \gamma_3\}$, and that
  $\Delta = \{\delta_1, \delta_2, \delta_3\}$, so that $\delta_i$ is
  dual to $\gamma_i$. Assume that we switch $\gamma_2$. Consider the
  surface $S$ (as in Section~\ref{sec:complexes}) obtained as the
  complement of the cut system $\{\gamma_1, \gamma_3\}$. Then both
  $\gamma_2, \delta_2$ are contained in $S$ and intersect in two points. 

  The dual curve $\delta_1$ defines two waves $w,w'$ with respect to
  $X$.  Consider $w$, and note that it intersects $\delta_2$ in two
  points (compare Figure~\ref{fig:cleanup-move}). There are two ways
  of surgering $w$ in the direction of $\delta_2$, i.e. replacing a
  subarc of $w$ by a subarc of $\delta_2$. Exactly one of them yields
  an essential wave, which we denote by $v$. Note that $v$ has the
  same endpoints as $w$. We define $v'$ similarly for the other wave
  $w'$.  We define $c(\delta_1)$ to be the curve $v \cup v'$. It
  intersects $\gamma_1$ in two points by construction, and is indeed
  nonseparating since it defines admissible waves (compare
  Lemma~\ref{lem:wave-description}).

  To see that the map $c$ commutes with Dehn twists about $\gamma_1$,
  it suffices to note that such a Dehn twist can be supported in a
  small neighbourhood of $\gamma_1$, and therefore the assignment of $v, v'$ 
  to $\delta_1$ commutes with Dehn twists by construction.
\end{proof}

\begin{definition}
  Suppose that $X \in \CP_2$ is a pants decomposition, that $\Delta$
  is a dual system, and $\gamma \in X$ is given. We then say that
  \emph{$(X', \Delta')$ is obtained from $(X, \Delta)$ by switching
    $\gamma$} if the following hold:
  \begin{enumerate}[i)]
  \item $X' = X\cup\{\delta\}\setminus\{\gamma\}$.
  \item The dual curve to $\delta$ is $\gamma$.
  \item The dual curves to both other $\delta' \in X'$ are obtained from
    the dual curves in $\Delta$ by the map $c$ from Lemma~\ref{lem:cleanup}.
  \end{enumerate}
\end{definition}
We record the following immediate corollary of the uniqueness
statement in Lemma~\ref{lem:cleanup}, which states that twisting about a curve different 
from $\gamma$ commutes with switching $\gamma$.
\begin{corollary}\label{cor:twist-commutes}
  Suppose that $X \in \CP_2$ is given, $\gamma, \gamma'$ are two curves in $X$. If
  $(X', \Delta')$ is obtained from $(X, \Delta)$ by switching $\gamma$, then
  $(X', T_{\gamma'}\Delta')$ is obtained from $(X, T_{\gamma'}\Delta)$ by switching $\gamma$.
\end{corollary}

We can now define our geometric model for the handlebody group of genus $2$.
\begin{definition}
  The graph $\CPT_2$ has vertices corresponding to pairs $(X, \Delta)$, where 
  $X$ is a vertex of $\CP_2$, and $\Delta$ is a dual system to $X$. There are
  two types of edges:
  \begin{description}
  \item[Twist] Suppose that $X = \{\gamma_1, \gamma_2, \gamma_3\}$ is a vertex
    of $\CP_2$, and $\Delta$ is a dual system. Then we join, for any $i$
    \[ (X, \Delta) \quad\mbox{and}\quad (X, T_{\gamma_i}^{\pm1}\Delta) \]
    by edges $e_i^{\pm}$. We call these \emph{twist edges} and say that
    $\gamma_i$ is \emph{involved in $e_i^{\pm}$}.

    If $\gamma$ is a curve that is involved in two oriented twist
    edges $e, e'$ we say that it is \emph{involved with consistent
      orientation} if the corresponding Dehn twist has the same sign in
    both cases.
  \item[Switch] 
    Suppose that $X \in \CP_2$ is a vertex, and $\Delta$ is a dual system to $X$.
    Suppose further that $(X', \Delta')$ is obtained 
    from $(X, \Delta)$ by switching some $\gamma \in X$. 

    We then connect $(X, \Delta)$ and $(X', \Delta')$ by an edge
    $e$. We say that $e$ is a \emph{switch edge}, and that it
    \emph{corresponds to the edge between $X$ and $X'$ in $\CP_2$}.
  \end{description}
\end{definition}
\begin{proposition}\label{prop:model}
  The handlebody group $\Han_2$ acts on $\CPT_2$ properly
  discontinuously and cocompactly.  
\end{proposition}
\begin{proof}
%
%
  The quotient of
  $\CPT_2$ by the handlebody group is finite, since $\Han_2$
  acts transitively on the vertices of $\CP_2$, and the group
  generated by Dehn twists about $X$ act transitively on dual systems
  of $X$. Since for a vertex $(X, \Delta)$ the union $X \cup \Delta$
  cuts the surface into simply connected regions, the stabilizer of
  any vertex of $\CPT_2$ is finite. 
\end{proof}

\subsection{Cubical Structure}
\label{sec:cubes}
In this section we will turn $\CPT_2$ into the $1$--skeleton of a $\mathrm{CAT}(0)$
cube complex. 

In order to do so, we will glue in two types of cubes into $\CPT_2$.
For the first, fix some $X \in \CP_2$, and consider the subgraph
$\CPT_2(X)$ spanned by those vertices whose corresponding pair has $X$
as its first entry. By definition, any edge in $\CPT_2(X)$ is a twist
edge, and in fact $\CPT_2(X)$ is isomorphic as a graph to the standard
Cayley graph of $\mathbb{Z}^3$. We call the subgraphs $\CPT_2(X)$
\emph{twist flats}. We then glue standard Euclidean cubes to make
$\CPT_2(X)$ the $1$-skeleton of the standard integral cube complex
structure of $\mathbb{R}^3$. We call these cubes \emph{twist cubes}.

The second kind of cubes will involve switch edges, and to describe them
we first need to understand all the switch edges corresponding
to a given edge between $X, X'$ in $\CP_2$. Let $\{\alpha_1, \alpha_2\} = X\cap X'$ be the two
curves that the two pants decompositions have in common, and suppose
$\gamma$ is switched to $\gamma'$. Then the possible switch edges will
join vertices
\[ \left( (\alpha_1, \alpha_2, \gamma), (\delta_1, \delta_2, \gamma')
\right) \quad\mbox{to}\quad \left( (\alpha_1, \alpha_2, \gamma'),
  (c(\delta_1), c(\delta_2), \gamma) \right) \] Note that since
$\delta_1, \delta_2$ are unique up to Dehn twists about $\alpha_1,
\alpha_2$ (Lemma~\ref{lem:dualsystem}), and the map $c$ commutes with twists
(Lemma~\ref{lem:cleanup}), we conclude that the switch edges corresponding to
the edge between $X, X'$ are exactly the edges between 
\[ \left( (\alpha_1, \alpha_2, \gamma), (T_{\alpha_1}^{n_1}\delta_1,
  T_{\alpha_2}^{n_2}\delta_2, \gamma') \right) \quad\mbox{and}\quad
\left( (\alpha_1, \alpha_2, \gamma'), (T_{\alpha_1}^{n_1}c(\delta_1),
  T_{\alpha_2}^{n_2}c(\delta_2), \gamma) \right) \]
for any $n_1, n_2$. Hence, in $\CPT_2$, there is a copy of the $1$--skeleton of a $3$--cube with vertices
\[ (X, \Delta), (X, T_{\alpha_1}\Delta), (X, T_{\alpha_2}\Delta), (X,
T_{\alpha_1}T_{\alpha_2}\Delta),\]\[ (X', \Delta'), (X',
T_{\alpha_1}\Delta'), (X', T_{\alpha_2}\Delta'), (X',
T_{\alpha_1}T_{\alpha_2}\Delta'), \]
 and we glue in a \emph{switch
  cube} at this $1$--skeleton. Similarly, we glue in three more switch
cubes for the different possibilities of replacing $T_{\alpha_1}$
and/or $T_{\alpha_2}$ by their inverses.

For later reference, observe that by construction any square in our cube complex
has either only twist edges, or exactly two nonadjacent switch edges in its boundary.
This fairly immediately implies the following.
\begin{proposition}\label{prop:link}
  The link of any vertex in the cube complex $\CPT_2$ is
  a flag simplicial complex.
\end{proposition}
\begin{proof}
  Since the link is a $2$-dimensional simplicial complex, we only have
  to check that any boundary of a triangle in the $1$--skeleton of the
  link bounds a triangle in the link. Vertices in the link correspond
  to edges in $\CPT_2$ and are therefore of twist or switch
  type. Edges in the link are due to squares in the cubical structure,
  and by the remark above any edge has either both endpoints of twist
  type, or exactly one of switch type. Thus, there are only two types
  of triangles in the link: those were all three vertices are of twist
  type, and those where exactly one of the vertices is of switch type.
  But, in both of these cases, the three corresponding edges in
  $\CPT_2$ are part of a common twist or switch cube, and therefore
  the desired triangle in the link exists.
\end{proof}

\begin{proposition}
  The cube complex $\CPT_2$ is connected and simply-connected.
\end{proposition}
\begin{proof}
  The fact that $\CPT_2$ is connected is an easy consequence of the
  fact that the tree $\CP_2$ is connected, and each twist flat
  $\CPT_2(X)$ is connected as well.

  Now, suppose that $g(i) = (X_i,
  \Delta_i)$ is a simplicial loop in $\CPT_2$. Then, the path $X_i$ is
  a loop in $\CP_2$, and since the latter is a tree, it has
  backtracking. Thus we can write $g$ as a concatenation
  \[ g = g_1 * \sigma_1 * \tau * \sigma_2 * g_2 \] where $\sigma_1,
  \sigma_2$ are two switch edges corresponding to an edge from some
  vertex $X$ to another vertex $X'$, and from $X'$ back to $X$,
  respectively, and $\tau$ is a path consisting only of twist
  edges. In fact, in order for $\sigma_1 * \tau * \sigma_2$ to be a
  path, the total twisting about the curve $\{\alpha\} = X'\setminus
  X$ has to be zero. Since the twist flat $\CPT_2(X')$ is homeomorphic
  to $\RR^3$ in our cubical structure, we may therefore homotope the
  path so that $\tau$ does not twist about $\alpha$ at all.

  Next, consider the first twist edge $t_1$ in $\tau$. Then,
  $\sigma_1*t_1$ are two sides of a square in a switch cube, and thus $g$ is
  homotopic to a path 
  \[ g = g_1 * t_1' * \sigma'_1 * \tau' * \sigma_2 * g_2 \] where now
  $\tau'$ has strictly smaller length than $\tau$, and $\sigma_1'$ and
  $\sigma_2$ still correspond to opposite orientations of the same
  edge in $\CP_2$.  By induction, we can reduce the length of $\tau'$
  to zero, in which case $g$ will have backtracking. An induction on
  the length of $g$ then finishes the proof.
\end{proof}
Hence, using Gromov's criterion (e.g. \cite[Chapter~II, Theorem~5.20]{BH})
we conclude:
\begin{corollary}\label{cor:cat0}
  $\CPT_2$ is a $\mathrm{CAT}(0)$ cube complex.
\end{corollary}

\begin{remark} By our construction, the genus $2$ handlebody group
  acts by semisimple isometries on a complete ${\rm CAT}(0)$-cube complex
  of dimension $3$. This should be contrasted to the main result of 
  \cite{B} which shows that any action of the mapping class group
  of a closed surface of genus $g\geq 2$ on a complete
  ${\rm CAT}(0)$-space
  of dimension less than $g$ fixes a point.
\end{remark}

\begin{corollary}\label{cor:biauto}
  The genus $2$ handlebody group $\Han_2$ is biautomatic, and has a quadratic
  isoperimetric inequality.
\end{corollary}
\begin{proof}
  From \cite[Corollary~8.1]{Swiatkowski} we conclude biautomaticity,
  since $\Han_2$ acts properly discontinuously and cocompactly on the
  $\mathrm{CAT}(0)$ cube complex $\CPT_2$ (Proposition~\ref{prop:model} and
  Corollary~\ref{cor:cat0}). This implies that $\Han_2$ has at most
  quadratic Dehn function \cite{BGSS, ET}.
  Observe that since $\Han_2$ contains copies of $\mathbb{Z}^2$
  (generated by Dehn twists about disjoint meridians) it is not
  hyperbolic, and therefore its Dehn function cannot be sub-quadratic
  \cite{Gromov}.
\end{proof}

Using Proposition 1 of \cite{CMV} we also conclude

\begin{corollary}\label{cor:haagerup}
  The genus $2$ handlebody group has the Haagerup property.
  \end{corollary}

\subsection{Other geometric consequences}
\label{sec:other-cons}
The geometric model $\CPT_2$ for the genus $2$ handlebody group can
also be used to conclude other facts about $\Han_2$. For example, we
have the following distance estimate in $\CPT_2$, which should be compared
to the Masur-Minsky distance formula for the surface mapping class group from \cite{MM2}.
\begin{proposition}\label{prop:distance-estimate}
  There are constants $c, C>0$ so that for all pairs of vertices
  $(X, \Delta_X), (Y, \Delta_Y) \in \CPT_2$ we have
  \[ d_{\CPT_2} ((X,\Delta_X), (Y,\Delta_Y)) \simeq_c d_{\CP_2}(X, Y) +
    \sum_\alpha [d_\alpha(X\cup \Delta_X, Y\cup \Delta_Y)]_C \]
  where the sum is taken over all non-separating meridians $\alpha$.
  Here, $\simeq_c$ means that the equality holds up to a (uniform)
  multiplicative and additive constant $c$, and $[\cdot]_C$ means that the
  term only appears of the argument is at least $C$.
\end{proposition}
\begin{proof}
  Consider a geodesic $g:[0,l] \to \CPT_2$ joining $(X, \Delta_X)$ to $(Y, \Delta_Y$)
  in $\CPT_2$. We need to estimate the length $l$ of $g$. First, we claim that
  the projection of $g$ to $\CP_2$ is a path without backtracking
  in the tree $\CP_2$ (but possibly with intervals on which it is
  constant).

  Namely, suppose that this is not the case. Then, as the projection $X_i$ of $g$ 
  backtracks, we can write $g$ as a concatenation
  \[ g = g_1 * \sigma * \tau * \sigma' * g_2 \] where $\sigma,
  \sigma'$ are two switch edges corresponding to opposite orientations of the same edge in $\CP_2$, and
  $\tau$ is a path just consisting of twist edges. If $\sigma$
  switches a curve $\gamma$, note that we may assume that $\tau$ does
  not twist about $\gamma$. Namely, the total twisting about $\gamma$
  has to be zero in order for $\sigma'$ to be able to follow
  $\tau$, and therefore any twists about $\gamma$ can be canceled
  without changing the length or the projection to $\CP_2$ of $g$.
  
  However, now the twists $\tau$ can be moved to the end of $g_1$ by
  Corollary~\ref{cor:twist-commutes} without changing the length of
  $g$ or its endpoints. However, after this modification $g$ has
  backtracking, which contradicts the fact that it is a geodesic.
  Similarly, arguing as above, we see that in a geodesic $g$ all the
  twist edges involving a given curve $\alpha$ need to have consistent
  orientation, as otherwise the geodesic could be shortened.

  Using Corollary~\ref{cor:twist-commutes} again, we may also assume that
  all twist edges involving the same curve $\alpha$ are adjacent in $g$,
  and appear immediately after $\alpha$ has become a curve in one of
  the $X_i$.

  \smallskip Now, the number of
  switch edges in $g$ is exactly $d_{\CP_2}(X, Y)$. It
  therefore suffices to argue that the number of twist edges can be
  estimated by the right-hand side of the equality in the proposition.

  Fix some non-separating meridian $\alpha$. If $\alpha$ never appears
  in $X_i$, then
  by Proposition~\ref{prop:annulus-projection-constant} the projection
  into the annulus around $\alpha$ is coarsely constant. Hence, by
  choosing $C$ large enough, these projections will not contribute to
  the sum in the statement of the proposition.

  If $\alpha$ does appear in $X_i$, then by
  Corollary~\ref{cor:twist-control}, it appears exactly for $i_0 \leq
  i \leq i_1$, and the projection before $i_0$ and after $i_1$ is
  coarsely constant. If $\alpha$ appears at $i_0$, and $g$ performs $n$ twists about $\alpha$ at this time, the
  projection $\pi_\alpha$ changes by a distance of $n$.  For any
  subsequent switch edge corresponding to $X_{i_0+1}, \ldots,
  X_{i_1}$, the projection can only change by a uniformly small amount
  in each cleanup move (given by Lemma~\ref{lem:cleanup}) as the
  corresponding dual curves intersect in uniformly few points. In
  conclusion, we have that $d_\alpha(X\cup \Delta_X, Y\cup \Delta_Y)$
  differs from $n$, where $n$ is the length of the twist segment in
  $g$ corresponding to $\alpha$, by at most the
  length $e_\alpha = i_1 - i_0$. However, observe that
  \[ \sum_{\alpha \in X_i, i=1,\ldots, k} e_\alpha \leq 3k \] since
  any vertex in the geodesic $(X_i)$ can be in at most three of the
  intervals whose lengths are counted as the $e_\alpha$. Hence, the
  sum of the error terms $e_\alpha$ is bounded by $d_{\CP_2}(X, Y)$,
  showing the proposition.
\end{proof}

From this distance formula, we immediately see the following: 
\begin{corollary}\label{cor:meridian-stabilisers}
  The stabilizer of a nonseparating meridian $\delta$ in $V_2$ is
  undistorted in $\Han_2$.
\end{corollary}
\begin{corollary}\label{cor:BBF}
  There is a quasi-isometric embedding of $\Han_2$ into a product
  of quasi-trees.
\end{corollary}
\begin{proof}
  From Proposition~\ref{prop:distance-estimate}, and the Masur-Minsky distance
  formula we see that the map 
  \[ \mathcal{H}_2 \to \CP_2 \times \mathrm{Mcg}(\Sigma_2) \] is a
  quasi-isometric embedding. By Theorem~\ref{thm:CP-tree}, the
  factor $\CP_2$ is already a tree. By the main result of \cite{BBF}, we can embed the
  second factor isometrically into a product of quasi-trees.
\end{proof}
\begin{remark}\label{rem:actual-bbf}
  In fact, by arguing exactly like in \cite{BBF}, $\Han_2$ embeds into
  $\CP_2 \times Y_1\times\cdots\times Y_k$, where each $Y_k$ is a
  quasi-tree of metric spaces coming from applying the main
  construction of \cite{BBF} to the set of all non-separating
  meridians, and projections into annuli around them. Since we do not
  need this precise result, we skip the proof.
\end{remark}

\bibliographystyle{math}
\bibliography{dehn}

\bigskip
\noindent Ursula Hamenst\"adt\\
Rheinische Friedrich-Wilhelms Universit\"at Bonn, Mathematisches Institut\\
Endenicher Allee 60, 53115 Bonn\\
Email: ursula@math.uni-bonn.de

\smallskip
\noindent Sebastian Hensel\\
Mathematisches Institut der Universit\"at M\"unchen\\
Theresienstra\ss e 39, 80333 M\"unchen\\
Email: hensel@math.lmu.de

\end{document}